\documentclass[11pt, reqno]{amsart}
\usepackage{amsmath,amsthm,amsfonts,amssymb,amscd}
\usepackage{epsfig}
\input{epsf.tex}

\newtheorem{theorem}{Theorem}[section]
\newtheorem{proposition}[theorem]{Proposition}
\newtheorem{lemma}[theorem]{Lemma}
\newtheorem{corollary}[theorem]{Corollary}
\newtheorem{define}[theorem]{Definition}
\newtheorem{remark}[theorem]{Remark}

\def\Empty{}

\catcode`\@=11
\def\section{\@startsection {section}{1}{\z@}{-3.5ex plus -1ex minus
-.2ex}{2.3ex plus .2ex}{\large\bf}}

\def\fnum@figure{{\small Figure \thefigure}}
\def\fakefigure{\def\@captype{figure}}

\long\def\@makecaption#1#2{
    \vskip 10pt
    \def\FCap{#2} \def\NoCap{\ignorespaces}
    \ifx \FCap\NoCap
       \setbox\@tempboxa\hbox{#1}  % This is to avoid the damn colon.
      \else
       \setbox\@tempboxa\hbox{#1: \small \it #2}
    \fi
    \ifdim \wd\@tempboxa >\hsize   % IF longer than one line:
        \unhbox\@tempboxa\par      %   THEN set as ordinary paragraph.
      \else                        %   ELSE  center.
        \hbox to\hsize{\hfil\box\@tempboxa\hfil}
    \fi}

\def\@oddhead{\hbox{}\rightmark \hfil \rm\thepage}% Right heading.
\def\sectionmark#1{\markright {\sc{\ifnum \c@secnumdepth >\z@
      \S\thesection.\hskip 1em\relax \fi #1}}}

\catcode`\@=12

\def\oplabel#1{
  \def\OpArg{#1} \ifx \OpArg\Empty {} \else
  	\label{#1}
  \fi}

\newlength{\saveu}

\catcode`\@=11

\newcommand{\wx}{\mbox{$\tilde x$}}
\newcommand{\wy}{\mbox{$\tilde y$}}
\newcommand{\wz}{\mbox{$\tilde z$}}

\newcommand{\cp}{\mbox{${\mathcal P}$}}

\newcommand{\cc}{\mbox{$\mathcal C$}}
\newcommand{\cT}{\mbox{$\mathcal T$}}

\newcommand{\oo}{\mbox{$\mathcal O$}}
\newcommand{\oos}{\mbox{${\mathcal O}^s$}}
\newcommand{\oou}{\mbox{${\mathcal O}^u$}}

\newcommand{\rrrr}{\mbox{${\bf R}$}}

\newcommand{\mi}{\mbox{$\widetilde M$}}
\newcommand{\wA}{\mbox{$\widetilde A$}}

\newcommand{\ls}{\mbox{$\Lambda^s$}}
\newcommand{\lu}{\mbox{$\Lambda^u$}}

\newcommand{\wls}{\mbox{$\widetilde \Lambda^s$}}
\newcommand{\wlu}{\mbox{$\widetilde \Lambda^u$}}

\newcommand{\wws}{\mbox{$\widetilde W^s$}}
\newcommand{\wwu}{\mbox{$\widetilde W^u$}}

\newcommand{\wwp}{\mbox{$\widetilde \Phi$}}

\newcommand{\gsg}{\mbox{${\mathcal G}^s$}}
\newcommand{\gug}{\mbox{${\mathcal G}^u$}}

\newcommand{\wt}{\mbox{$\widetilde T$}}

\newcommand{\fol}{\mbox{$\mathcal F$}}

\newcommand{\fn}{\mbox{$\widetilde {\mathcal F}$}}

\newcommand{\ws}{\mbox{$\widetilde  W^s$}}
\newcommand{\wu}{\mbox{$\widetilde  W^u$}}

\catcode`\@=12

\def\centeredepsfbox#1{\centerline{\epsfbox{#1}}}

\begin{document}

\title[\centerline{Totally periodic pseudo-Anosov flows in graph manifolds}]{Classification and Rigidity of totally periodic pseudo-Anosov flows in graph manifolds}%\footnote{Preliminary version. Comments will be appreciated.}}

\author{Thierry Barbot}

\author{S\'{e}rgio R. Fenley}
%\thanks{Reseach partially supported by NSF grant DMS-0305313.}

\address{Thierry Barbot\\ Universit\'e d'Avignon et des pays de Vaucluse\\
LANLG, Facult\'e des Sciences\\
33 rue Louis Pasteur\\
84000 Avignon, France.}

\email{thierry.barbot@univ-avignon.fr}

\address{S\'ergio Fenley\\Florida State University\\
Tallahassee\\FL 32306-4510, USA \ \ and \ \
Princeton University\\Princeton\\NJ 08544-1000, USA}

\email{fenley@math.princeton.edu}

 \email{}
\maketitle

%\centerline{PRELIMINARY VERSION}
\vskip .2in

{\small
\noindent
{{\underline {Abstract}} $-$
In this article we analyze totally periodic pseudo-Anosov
flows in graph three manifolds. This means
that in
each Seifert fibered piece of the torus decomposition,
the free homotopy class of regular fibers has a finite power which is also a finite power of the
free homotopy class of a closed orbit of the flow.
We show that each such flow is topologically
equivalent to one of the  model pseudo-Anosov flows which we previously
constructed in \cite{bafe}. A model
pseudo-Anosov flow is obtained by glueing
standard neighborhoods of Birkhoff annuli and perhaps
doing Dehn surgery on certain orbits.
We also show that two model flows on the same graph manifold are isotopically equivalent (ie. there is a isotopy of $M$
mapping the oriented orbits of the first flow to the oriented orbits of the second flow)
if and only if they have the same topological and
dynamical data in the
collection of standard neighborhoods of the Birkhoff annuli.
}
}

%\tableofcontents

\section{Introduction}

Pseudo-Anosov flows are extremely common amongst
three manifolds, for example\footnote{We also mention a recent work in
progress by F. B\'eguin, C. Bonatti and Bin Yu,
constructing a wide family of new Anosov flows; which can be seen as an extension
of the construction in \cite{bafe} (\cite{BBB}).}: \  1) Suspension
pseudo-Anosov flows \cite{Th1,Th2,Th3}, 2) Geodesic flows in
the unit tangent bundle of negatively curved surfaces \cite{An},
3) Certain flows transverse to foliations in
closed atoroidal manifolds \cite{Mo3,Cal1,Cal2,Cal3,Fe4};
\ flows obtained
from these by either 4) Dehn surgery on a closed orbit
of the pseudo-Anosov flow \cite{Go,Fr}, \ or 5) Shearing along tori
\cite{Ha-Th}; \ 6) Non transitive Anosov flows \cite{Fr-Wi} and
flows with transverse tori \cite{Bo-La}.

The purpose of this article is to analyse the question:
how many
essentially different pseudo-Anosov flows are there
in a manifold?
Two flows are essentially the same if they
are {\em topologically equivalent}.
This means that there is a homeomorphism
between the manifolds which sends orbits of the first
flow to orbits of the second flow preserving orientation along the orbits.
In this article, we will also consider the notion of {\em isotopic equivalence}, i.e.
a topological equivalence induced by an isotopy, that is, a homeomorphism
isotopic to the identity.

We will restrict to closed, orientable, toroidal manifolds. In particular
they are sufficiently large in the sense of \cite{waldlarge}, that is,
they have {\em incompressible surfaces} \cite{He,Ja}.
Manifolds with pseudo-Anosov flows are also irreducible
\cite{Fe-Mo}. It follows that these manifolds are Haken \cite{Ja}.
We have recently extended a result of the first author (\cite{Ba2}) to the case
of general pseudo-Anosov flows: if the ambient manifold
is Seifert fibered, then the flow is up to finite
cover topologically equivalent to a geodesic flow in
the unit tangent bundle of a closed hyperbolic surface
\cite[Theorem A]{bafe}.
In addition we also
proved that if the ambient manifold is a solvable three
manifold, then the flow is topologically equivalent to a
suspension Anosov flow \cite[Theorem B]{bafe}.
Notice that in both cases the flow
does not have singularities, that is, the type of the
manifold strongly restricts the type of pseudo-Anosov
that it can admit. This is in contrast with the
strong flexibility in the construction of pseudo-Anosov
flows $-$ that is because many flows are constructed
in atoroidal manifolds or are obtained by flow
Dehn surgery on the pseudo-Anosov flow, which
changes the topological type of the manifold.
Therefore in many constructions one 
cannot expect that the underlying manifold
is toroidal.

In this article we will mainly study pseudo-Anosov flows in
graph manifolds. A {\em graph manifold} is an irreducible
three manifold which is a union of Seifert fibered
pieces. In a previous article \cite{bafe} we produced
a large new class of examples in graph manifolds. These
flows are totally periodic. This means that each
Seifert piece of the torus decomposition of
the graph manifold is {\em periodic}, that is,
up to finite powers, a regular fiber is freely homotopic to a closed
orbit of the flow. More recently, Russ Waller
\cite{Wa} has been studying how common these 
examples are, that is, the existence question for these type
of flows. He showed that these flows are as common as they
could be (modulo the necessary conditions).

In this article we will analyse the question of
%uniqueness of such flows or
the classification and rigidity
of such flows. In order to state and understand the results
of this article we need to introduce the fundamental
concept of a {\em Birkhoff annulus}.
A Birkhoff annulus is an a priori only immersed annulus, so that the
boundary is a union of closed orbits of the flow and the
interior of the annulus is transverse to the flow.
For example consider the geodesic flow of a closed, orientable
hyperbolic surface. The ambient manifold is the unit
tangent bundle of the surface. Let $\alpha$ be an oriented
closed geodesic - a closed orbit of the flow - and consider
a homotopy that turns the angle along $\alpha$ by $\pi$.
The image of the homotopy from $\alpha$ to the same geodesic
with opposite orientation is a Birkhoff annulus
for the flow in the unit tangent bundle.
If $\alpha$ is not embedded then the Birkhoff annulus
is not embedded.
In general Birkhoff annuli are not embedded, particularly
in the boundary.
A Birkhoff annulus is transverse to the flow in its interior,
so it has induced stable and unstable foliations.
The Birkhoff annulus is {\em elementary} if these foliations
in the interior have no closed leaves.

In \cite[Theorem F]{bafe} we proved the following basic result about 
the relationship of a pseudo-Anosov flow and a periodic
Seifert piece $P$: there is a spine $Z$ for
$P$ which is a connected union of finitely many
elementary Birkhoff annuli. In addition the union of the
interiors of the Birkhoff annuli is embedded
and also disjoint
from the closed orbits in $Z$. These closed orbits, boundaries
of the Birkhoff annuli in $Z$, are called {\em vertical periodic orbits.}
The set $Z$ is a deformation
retract of $P$, so $P$ is isotopic to a  small
compact neighborhood $N(Z)$ of $Z$.
In general the Birkhoff annuli are not embedded in $Z$: 
it can be that the two boundary components
of the same Birkhoff annulus are the same vertical periodic
orbit of the flow. It can also occur that the
annulus wraps a few times around one of
its boundary orbits.
These are not exotic occurrences, but rather fairly common.
For every vertical periodic orbit $\alpha$ in $N(Z)$, the local stable leaf
of $\alpha$ is a finite union of annuli, called {\em stable vertical annuli,}
tangent to the flow, each realizing a homotopy between (a power of)
$\alpha$ and a closed loop
in $\partial N(Z)$. One defines similarly {\em unstable} vertical annuli in $N(Z)$.

We first analyse periodic Seifert pieces. The first theorems (Theorem A and B) are valid
for any closed orientable manifold $M$, not necessarily a graph manifold.
The first result is (see Proposition \ref{disjspin}):

\vskip .1in
\noindent
{\bf {Theorem A}} $-$ Let $\Phi$ be a pseudo-Anosov flow in
$M^3$. If $\{ P_i \}$ is the (possibly empty) collection of
periodic Seifert pieces of the torus decomposition of $M$,
then the spines $Z_i$ and neighborhoods $N(Z_i)$ can be
chosen to be pairwise disjoint.
\vskip .1in

We prove that the $\{ Z_i \}$ can be chosen pairwise disjoint.
Roughly this goes as follows: we show that
the vertical periodic orbits in $Z_i$ cannot intersect $Z_j$ for $j \not = i$,
because fibers in different Seifert pieces cannot have
common powers which are freely homotopic. We also show that
the possible interior
intersections are null homotopic and can be isotoped away.

The next result (Proposition \ref{transtor}) shows that the boundary of the pieces
can be put in good position with respect to the flow:

\vskip .1in
\noindent
{\bf {Theorem B}} $-$ Let $\Phi$ be a pseudo-Anosov flow and $P_i, P_j$
be periodic Seifert pieces with a common boundary torus $T$.
Then $T$ can be isotoped to a torus transverse to the flow.
\vskip .1in

The main property used to prove this result is that regular fibers restricted
to both sides of $T$ (from $P_i$ and $P_j$) cannot represent the
same isotopy class in $T$.

Finally we prove the following (Proposition \ref{goodposition}):
\vskip .1in
\noindent
{\bf {Theorem C}} $-$ Let $\Phi$ be a totally periodic pseudo-Anosov
flow with periodic Seifert pieces $\{ P_i \}$. Then
neighborhoods $\{ N(Z_i) \}$
of the spines $\{ Z_i \}$
can be chosen so that their union is $M$ and they have
pairwise disjoint
interiors. In addition each boundary component of every
$N(Z_i)$ is transverse to the flow. Each $N(Z_i)$ is flow
isotopic to an arbitrarily small neighborhood
of $Z_i$.
\vskip .1in

We stress that for general periodic pieces it is not true
that the boundary of $N(Z_i)$ can be isotoped to be
transverse to the flow. There are some simple examples
as constructed in \cite{bafe}. The point here is that we assume that
{\em all} pieces of the JSJ decomposition are periodic Seifert pieces.

Hence, according to Theorem C, totally periodic pseudo-Anosov flow are obtained
by glueing along the bondary  a collection of
small neighborhoods $N(Z_i)$ of the spines.
There are several ways to perform this glueing which lead to pseudo-Anosov flows.
The main result of this paper is that the resulting pseudo-Anosovs flow are
all topologically equivalent one to the other. More precisely (see section~\ref{sub.theoremD}):

\vskip .1in
\noindent
{\bf {Theorem D}} $-$ Let $\Phi$, $\Psi$ be two totally periodic pseudo-Anosov
flows on the same orientable graph manifold $M$. Let $P_i$ be the Seifert pieces of
$M$, and let $Z_i(\Phi)$, $Z_i(\Psi)$ be spines of $\Phi$, $\Psi$ in $P_i$.
Then, $\Phi$ and $\Psi$ are topologically equivalent if and only if there is a homeomorphism of $M$ mapping
the collection of spines $\{ Z_i(\Phi) \}$ onto the
collection $\{ Z_i(\Psi) \}$ and preserving the orientations of the vertical periodic orbits induced by the flows.
\vskip .1in

Theorem D is a consequence of the following Theorem, more technical but
slightly more precise (see section~\ref{pro.theoremD'}):

\vskip .1in
\noindent
{\bf {Theorem D'}} $-$ Let $\Phi$, $\Psi$ be two totally periodic pseudo-Anosov
flows on the same orientable graph manifold $M$. Let $P_i$ be the Seifert pieces of
$M$, and let $Z_i(\Phi)$, $Z_i(\Psi)$ be spines of $\Phi$, $\Psi$ in $P_i$, with tubular
neighborhoods $N(Z_i(\Phi))$, $N(Z_i(\Psi))$ as in the statement of Theorem C.
Then, $\Phi$ and $\Psi$ are isotopically equivalent if and only if,
after reindexing the collection $\{ N(Z_i(\Psi)) \}$,
 there is an isotopy in $M$ mapping every spine $Z_i(\Phi)$ onto $Z_i(\Psi)$, mapping every stable/unstable vertical annulus of $\Phi$ in $N(Z_i(\Phi))$ to a stable/vertical annulus in $N(Z_i(\Psi))$ and preserving the orientations of the vertical periodic orbits induced by the flows.
\vskip .1in

The main ideas of the proof are as follows. It is easy to show that, if the two flows
are isotopical equivalent, the isotopy maps every $Z_i(\Phi)$ onto a spine $Z_i(\Psi)$
of $\Psi$, every $N(Z_i(\Phi))$ onto a neighborhood $N(Z_i(\Psi))$, so that vertical stable/unstable annuli and the orientation of vertical periodic orbits are preserved.

Conversely, assume that up to isotopy $\Phi$ and $\Psi$ admit the same decomposition
in neighborhoods $N(Z_i)$ of spines $Z_i$, so that they share exactly the same oriented vertical periodic orbits and the same stable/unstable vertical annuli.
Consider all the lifts to the universal cover of the tori
in $\partial N(Z_i)$ for all $i$.
This is a collection $\cT$ of properly embedded topological
planes in $\mi$, which is transverse to the lifted
flows $\wwp$ and $\widetilde{\Psi}$. We show that an orbit of $\wwp$ or
$\widetilde{\Psi}$ (if not the lift of a vertical periodic orbits)
is completely determined by its itinerary up to shifts:
the itinerary is the collection of planes
it intersects. One thus gets a map between orbits of
$\wwp$ and orbits of $\widetilde{\Psi}$. This extends to the lifts of the vertical periodic
orbits. This is obviously group equivariant.
The much harder step is to prove that this is
continuous, which we do using the exact
structure of the flows and the combinatorics.
Using this result we can then show that the
flow $\Phi$ is topologically equivalent to
$\Psi$. Since the action on the fundamental group level
is trivial,
this topological equivalence is homotopic to the identity,
hence, by a Theorem by Waldhausen (\cite{waldlarge}),
isotopic to the identity: it is an isotopic equivalence.

We then show
(section~\ref{sub.isotopequiv}) that for any totally periodic pseudo-Anosov flow $\Phi$
there is a model pseudo-Anosov flow as constructed in \cite{bafe} which has precisely
the same data $Z_i$, $N(Z_i)$ that $\Phi$ has. This proves the following:

\vskip .1in
\noindent
{\bf {Main theorem}} $-$ Let $\Phi$ be a totally
periodic pseudo-Anosov flow in a graph manifold
$M$. Then $\Phi$ is topologically equivalent to
a model pseudo-Anosov flow.
\vskip .1in

Model pseudo-Anosov flows are defined
by some combinatorial data (essentially, the data of some fat graphs and Dehn surgery
coefficients; see section~\ref{sub.construct} for more details) and some parameter
$\lambda$. A nice corollary of Theorem D' is that, up to isotopic equivalence,
the model flows actually do not depend on the choice of $\lambda$, nor on the choice of
the selection of the particular glueing map between the model periodic pieces.

In the last section, we make a few remarks on the action of the mapping class group of $M$
on the space of isotopic equivalence classes of totally periodic pseudo-Anosov flows on $M$.

\section{Background}

\noindent
{\bf {Pseudo-Anosov flows $-$ definitions}}

\begin{define}{(pseudo-Anosov flow)}
Let $\Phi$ be a flow on a closed 3-manifold $M$. We say
that $\Phi$ is a pseudo-Anosov flow if the following conditions are
satisfied:

%\begin{itemize}

%\item
- For each $x \in M$, the flow line $t \to \Phi(x,t)$ is $C^1$,
it is not a single point,
and the tangent vector bundle $D_t \Phi$ is $C^0$ in $M$.

- There are two (possibly) singular transverse
foliations $\ls, \lu$ which are two dimensional, with leaves saturated
by the flow and so that $\ls, \lu$ intersect
exactly along the flow lines of $\Phi$.
%These are like Anosov foliations off of the singular orbits.
%This is the topologically smooth behavior described above.

- There is a finite number (possibly zero) of periodic orbits $\{ \gamma_i \}$,
called singular orbits.
A stable/unstable leaf containing a singularity is homeomorphic
to $P \times I/f$
where $P$ is a $p$-prong in the plane and $f$ is a homeomorphism
from $P \times \{ 1 \}$ to $P \times \{ 0 \}$.
In addition $p$ is at least $3$.
%We stress that $p$ is an integer satisfying $p \geq 3$ $-$
%no $1$-prongs allowed.
%- In a stable leaf, $f$ contracts towards towards
%the prongs and in an unstable leaf it expands away
%from the prongs.
%Most of the time, except in \S , we restrict to $p$ at least $3$, that is, we do not allow
%$1$-prongs.

- In a stable leaf all orbits are forward asymptotic,
in an unstable leaf all orbits are backwards asymptotic.
\end{define}

Basic references for pseudo-Anosov flows are \cite{Mo1,Mo2} and
\cite{An} for Anosov flows. A fundamental remark is that the ambient manifold
supporting a pseudo-Anosov flow is necessarily irreducible - the
universal covering is homeomorphic to ${\bf R}^3$ (\cite{Fe-Mo}).
We stress that in our definition one prongs are not allowed.
There are however ``tranversely hyperbolic" flows with one prongs:

\begin{define}{(one prong pseudo-Anosov flows)}{}
A flow $\Phi$ is a one prong pseudo-Anosov flow in $M^3$ if it satisfies
all the conditions of the definition of pseudo-Anosov flows except
that the $p$-prong singularities  can also be
$1$-prong ($p = 1$).
\end{define}

\vskip .05in
\noindent
{\bf {Torus decomposition}}

Let $M$ be an irreducible closed $3$--manifold. If $M$ is orientable, it  has a unique (up to isotopy)
minimal collection of disjointly embedded incompressible tori such that each component of $M$
obtained by cutting along the tori is either atoroidal or Seifert-fibered \cite{Ja,Ja-Sh}
and the pieces are isotopically maximal with this property.
If $M$ is not orientable,
a similar conclusion holds; the decomposition has to be performed along tori, but also along
some incompressible embedded Klein bottles.

Hence the notion of maximal Seifert pieces in $M$ is well-defined up to isotopy. If $M$ admits
a pseudo-Anosov flow, we say that a Seifert piece $P$  is {\em periodic} if there is a
Seifert fibration on $P$ for which, up to finite powers, a regular
fiber is freely homotopic to a periodic orbit of $\Phi$. If not,
the piece is called {\em free.}

\vskip .05in
\noindent{\bf {Remark. }}
In a few circumstances, the Seifert fibration is not unique: it happens for example
when $P$ is homeomorphic to a twisted line bundle over the Klein bottle or
$P$ is $T^2 \times I$.
We stress out that our convention is to say that the Seifert piece is free
if
{\underline {no}} Seifert fibration in $P$ has fibers homotopic to a periodic orbit.

\vskip .1in
\noindent
{\bf {Orbit space and leaf spaces of pseudo-Anosov flows}}

\vskip .05in
\noindent
\underline {Notation/definition:} \
We denote by $\pi: \mi \to M$ the universal covering of $M$, and by $\pi_1(M)$ the fundamental group of $M$,
considered as the group of deck transformations on $\mi$.
The singular
foliations lifted to $\mi$ are
denoted by $\wls, \wlu$.
If $x \in M$ let $W^s(x)$ denote the leaf of $\ls$ containing
$x$.  Similarly one defines $W^u(x)$
and in the
universal cover $\ws(x), \wu(x)$.
Similarly if $\alpha$ is an orbit of $\Phi$ define
$W^s(\alpha)$, etc...
Let also $\wwp$ be the lifted flow to $\mi$.

\vskip .05in

%Let $\Phi$ be a pseudo-Anosov flow in $M^3$ closed.
We review the results about the topology of
$\wls, \wlu$ that we will need.
% which will be needed in the following
%sections to prove the main theorem.
We refer to \cite{Fe2,Fe3} for detailed definitions, explanations and
proofs.
The orbit space of $\wwp$ in
$\mi$ is homeomorphic to the plane $\rrrr^2$ \cite{Fe-Mo}
and is denoted by $\oo \cong \mi/\wwp$. There is an induced action of $\pi_1(M)$ on $\oo$.
Let

$$\Theta: \ \mi \ \rightarrow \ \oo \ \cong \ \rrrr^2$$

\noindent
be the projection map: it is naturally $\pi_1(M)$-equivariant.
%As the foliations $\fns, \fnu$
%are invariant under $\wwp$, they
%induce singular, transverse $1$-dim foliations
%$\fnso, \fnuo$ in $\oo$.
%The singular points  of $\fnuo$ are the same as those
%of $\fnso$.
If $L$ is a
leaf of $\wls$ or $\wlu$,
then $\Theta(L) \subset \oo$ is a tree which is either homeomorphic
to $\rrrr$ if $L$ is regular,
or is a union of $p$-rays all with the same starting point
if $L$ has a singular $p$-prong orbit.
%In particular every orbit in $L$ disconnects $L$.
The foliations $\wls, \wlu$ induce $\pi_1(M)$-invariant singular $1$-dimensional foliations
$\oos, \oou$ in $\oo$. Its leaves are $\Theta(L)$ as
above.
If $L$ is a leaf of $\wls$ or $\wlu$, then
a {\em sector} is a component of $\mi - L$.
Similarly for $\oos, \oou$.
If $B$ is any subset of $\oo$, we denote by $B \times \rrrr$
the set $\Theta^{-1}(B)$.
The same notation $B \times \rrrr$ will be used for
any subset $B$ of $\mi$: it will just be the union
of all flow lines through points of $B$.
We stress that for pseudo-Anosov flows there are at least
$3$-prongs in any singular orbit ($p \geq 3$).
For example, the fact that the orbit space in $\mi$ is
a $2$-manifold is not true in general if one allows
$1$-prongs.

\begin{define}
Let $L$ be a leaf of $\wls$ or $\wlu$. A slice of $L$ is
$l \times \rrrr$ where $l$ is a properly embedded
copy of the reals in $\Theta(L)$. For instance if $L$
is regular then $L$ is its only slice. If a slice
is the boundary of a sector of $L$ then it is called
a line leaf of $L$.
If $a$ is a ray in $\Theta(L)$ then $A = a \times \rrrr$
is called a half leaf of $L$.
%The closure is denoted by $\overline A = A \cup \gamma$ and
%its boundary is $\partial A = \gamma$.
If $\zeta$ is an open segment in $\Theta(L)$
it defines a {\em flow band} $L_1$ of $L$
by $L_1 = \zeta \times \rrrr$.
%\Theta^{-1}(\zeta)$.
%Let $\overline L_1$ be the closure of $L_1$ in $\mi$.
%If $\zeta$ is an open segment in $\Theta(L)$, then $\Theta^{-1}(\zeta)$
%is called a segment flow band of $L$.
We use the same terminology of slices and line leaves
for the foliations $\oos, \oou$ of $\oo$.
\end{define}

If $F \in \wls$ and $G \in \wlu$
then $F$ and $G$ intersect in at most one
orbit.

We abuse convention and call
a leaf $L$ of $\wls$ or $\wlu$ {\em periodic}
if there is a non trivial covering translation
$g$ of $\mi$ with $g(L) = L$. This is equivalent
to $\pi(L)$ containing a periodic orbit of $\Phi$.
In the same way an orbit
$\gamma$ of $\wwp$
is {\em periodic} if $\pi(\gamma)$ is a periodic orbit
of $\Phi$. Observe that in general, the stabilizer of an element $\alpha$
of $\oo$ is either trivial, or a cyclic subgroup of $\pi_1(M)$.

\vskip .2in
\noindent
{\bf {Perfect fits, lozenges and scalloped chains}}

Recall that a foliation $\fol$ in $M$ is $\rrrr$-covered if the
leaf space of $\fn$ in $\mi$ is homeomorphic to the real line $\rrrr$
\cite{Fe1}.

\begin{define}{(\cite{Fe2,Fe3})}\label{def:pfits}
Perfect fits -
Two leaves $F \in \wls$ and $G \in \wlu$, form
a perfect fit if $F \cap G = \emptyset$ and there
are half leaves $F_1$ of $F$ and $G_1$ of $G$
and also flow bands $L_1 \subset L \in \wls$ and
$H_1 \subset H \in \wlu$,
so that
%$F_0$ is regular on the side containing $L$,
%$G_0$ is regular on the side containing $H$ and:
%
%$$ \overline L_1 \cap \overline G_1 = \partial L_1 \cap \partial G_1,
%\ \ \overline L_1 \cap \overline H_1 = \partial L_1 \cap \partial H_1,
%\ \ \overline H_1 \cap \overline F_1 =
%\partial H_1 \cap \partial F_1,$$
the set

$$\overline F_1 \cup \overline H_1 \cup
\overline L_1 \cup \overline G_1$$

\noindent
separates $M$ and forms an a rectangle $R$ with a corner removed:
The joint structure of $\wls, \wlu$ in $R$ is that of
a rectangle with a corner orbit removed. The removed corner
corresponds to the perfect of $F$ and $G$ which
do not intersect.
%`$${\rm with} \ \ \  \overline L_1 \cap \overline G_1 \ \not = \ \emptyset,
%\ \ \overline L_1 \cap \overline H_1  \ \not = \ \emptyset \ \ {\rm and}
%\ \ \overline H_1 \cap \overline F_1 \ \not = \ \emptyset.$$
%\noindent
%Furthermore
%$$\forall \ S \in \wlu, \ \ \ \
%S \cap L_1  \not = \emptyset \ \Rightarrow
%S \cap F_1  \not = \emptyset \ \ \ (1) $$
%$${\rm and} \ \ \forall \ E \in \wls, \ \ \ \ E \cap H_1
%\not = \emptyset \ \Rightarrow
%E \cap G_1  \not = \emptyset \ \ \ (2).$$
\end{define}

We refer to fig. \ref{loz}, a for perfect fits.
%Implications $(1), (2)$ force
%equivalences (that is
%$S \cap L_1  \not = \emptyset \ \Leftrightarrow
%S \cap F_1  \not = \emptyset$ and the same for (2)).
%The set
%$\overline F_1 \cup \overline H_1 \cup
%\overline L_1 \cup \overline G_1$
%separates $\mi$.
%Let $A$ be the complementary region which does
%not contain $F - F_1$ in its closure.
%An important fact is that there are
%singularities of $\wwp$
%in $A$.
%Perfect fits produce ``ideal" rectangles, in the sense that
%even though $F$ and $G$ do not intersect, there is
%a product structure (of $\wls$ and $\wlu$) in the
%interior of $A$.
There is a product structure in the interior of $R$: there are
two stable boundary sides and two unstable boundary
sides in $R$. An unstable
leaf intersects one stable boundary side (not in the corner) if
and only if it intersects the other stable boundary side
(not in the corner).
We also say that the leaves $F, G$ are {\em asymptotic}.

%\begin{define}{\cite{Fe2,Fe4}}{}
%Given $p \in \mi$ (or $p \in \oo$), and a half leaf $H$
%of $\wu(p)$ defined by $\wwp_{\rrrr}(p)$, let
%
%$$\uo(H) \ = \ \{ F \in \hs  \ | \ F \cap
%H \not = \emptyset \} \ \subset \ \hs.$$
%
%\noindent
%Notice that $\ws(p) \not \in \uo(H)$.
%Let  also
%
%$${\mathcal L}^u(H)  \ =
%\ \bigcup  \ \{ \ p \ \in \mi \ \
%| \ \ p \ \in \ F \ \in \ \uo(H) \ \} \ \ \subset \ \ \mi.$$
%
%\noindent
%Then ${\mathcal L}^u(H) \subset \mi$
%and $\ws(p) \subset \partial {\mathcal L}^u(H)$.
%Similarly define
%$\so(L),
%{\mathcal L}^s(L)$ for a stable half leaf $L$.
%\end{define}

\begin{figure}
\centeredepsfbox{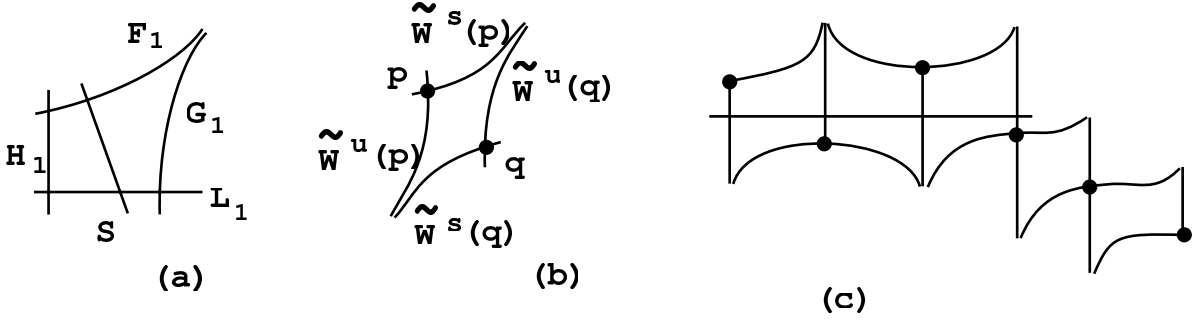}
\caption{a. Perfect fits in $\mi$,
b. A lozenge, c. A chain of lozenges.}
\label{loz}
\end{figure}

\begin{define}{(\cite{Fe2,Fe3})}\label{def:lozenge}
Lozenges - A lozenge $R$ is a region of $\mi$ whose closure
is homeomorphic to a rectangle with two corners removed.
More specifically two points $p, q$ define the corners
of a lozenge if there are half leaves $A, B$ of
$\ws(p), \wu(p)$ defined by $p$
and  $C, D$ half leaves of $\ws(q), \wu(q)$ defined by $p, q$, so
that $A$ and $D$ form a perfect fit and so do
$B$ and $C$. The region bounded by the lozenge
$R$ does not have any singularities.
%Let $p, q \in \mi$ and half leaves
%$L_p, H_p$ of $\ws(p), \wu(p)$ defined by $\wwp_{\rrrr}(p)$,
%half leaves $L_q, H_q$ of $\ws(q), \wu(q)$
%defined by $\wwp_{\rrrr}(q)$
%%so that:
%
%
%$${\mathcal L}^u(L_p) \ \cap \
%{\mathcal L}^s(H_q) \ \  = \ \
%{\mathcal L}^u(L_q) \ \cap \
%{\mathcal L}^s(H_p) \ \ \subset \ \mi$$
%
%
%\noindent
%Then this intersection
%is called a lozenge ${\mathcal A}$
%in $\mi$, see fig. \ref{loz}, b.
%The corners of the lozenge are $\wwp_{\rrrr}(p)$ and
%$\wwp_{\rrrr}(q)$ and ${\mathcal A}$
%is a subset of $\mi$.
The sides of $R$ are $A, B, C, D$.
The sides are not contained in the lozenge,
but are in the boundary of the lozenge.
There may be singularities in the boundary of the lozenge.
See fig. \ref{loz}, b.
\end{define}

%Sometimes we also refer to $p$ and $q$ as corners of the lozenge.

There are no singularities in the lozenges,
which implies that
$R$ is an open region in $\mi$.
%There may be singular orbits
%on the sides of the lozenge and the corner orbits.
%%also may be singular.
%The definition of a lozenge implies that
%$L_p, H_q$ form a perfect fit and so do
%$L_q, H_p$.
%This is an equivalent way to define a lozenge
%with corners $\wwp_{\rrrr}(p), \wwp_{\rrrr}(q)$.

%Intuitively a lozenge is a rectangle with two
%opposite corners turned into ideal points:
%each removed corner corresponds to a pair of leaves
%(the sides of the lozenge) forming a perfect fit.

Two lozenges are {\em adjacent} if they share a corner and
there is a stable or unstable leaf
intersecting both of them, see fig. \ref{loz}, c.
Therefore they share a side.
A {\em chain of lozenges} is a collection $\{ \cc _i \},
i \in I$, where $I$ is an interval (finite or not) in ${\bf Z}$;
so that if $i, i+1 \in I$, then
${\mathcal C}_i$ and ${\mathcal C}_{i+1}$ share
a corner, see fig. \ref{loz}, c.
Consecutive lozenges may be adjacent or not.
The chain is finite if $I$ is finite.

\begin{define}{(scalloped chain)}\label{def:scallop}
Let ${\mathcal C}$ be a chain of lozenges.
If any two
successive lozenges in the chain are adjacent along
one of their unstable sides (respectively stable sides),
then the chain is called {\em s-scalloped}
(respectively {\em u-scalloped}) (see
fig. \ref{pict} for an example of a $s$-scalloped chain).
Observe that a chain is s-scalloped if
and only if there is a stable leaf intersecting all the
lozenges in the chain. Similarly, a chain is u-scalloped
if and only if there is an unstable leaf intersecting
all the lozenges in the chain.
The chains may be infinite.
A scalloped chain is a chain that is either $s$-scalloped or
$u$-scalloped.
\end{define}

\begin{figure}
\centeredepsfbox{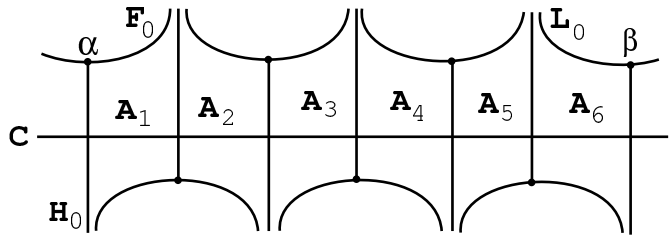}
\caption{
A partial view of a scalloped region.
Here $C, F_0, L_0$ are stable leaves,
so this is a s-scalloped region.}
\label{pict}
\end{figure}

For simplicity when considering scalloped chains we also include any half leaf which is
a boundary side of two of the lozenges in the chain. The union of these
is called a {\em {scalloped region}} which is then a connected set.

We say that two orbits $\gamma, \alpha$ of $\wwp$
(or the leaves $\ws(\gamma), \ws(\alpha)$)
are connected by a
chain of lozenges $\{ {\mathcal C}_i \}, 1 \leq i \leq n$,
if $\gamma$ is a corner of ${\mathcal C}_1$ and $\alpha$
is a corner of ${\mathcal C}_n$.

\begin{remark}
\label{rk.lozengeannulus}
{\em A key fact, first observed in \cite{Ba3}, and extensively used in \cite{bafe}, is the following:
the lifts in $\widetilde{M}$ of elementary Birkhoff annuli are precisely lozenges invariant by some cyclic subgroup of $\pi_1(M)$ (see \cite[Proposition $5.1$]{Ba3} for the case
of embedded Birkhoff annuli). It will also play a crucial role in the sequel.
More precisely: let $A$ be an elementary Birkhoff
annulus. We say that $A$ lifts to the
lozenge $C$ in $\mi$ if the interior
of $A$ has a lift which intersects
orbits only in $C$. It follows that
this lift intersects every orbit in $C$
exactly once and also that the two
boundary closed orbits of $A$ lift
to the
full corner orbits of $C$.
This uses the fact that a lozenge cannot be properly
contained
in another lozenge.}
\end{remark}

In particular
the following important property also follows: if $\alpha$ and $\beta$
are the periodic orbits in $\partial A$ (traversed in the
flow forward direction), then there are
{\underline {positive}} integers $n, m$ so that
$\alpha^n$ is freely homotopic
to $(\beta^m)^{-1}$. We emphasize the free homotopy between inverses.

\begin{remark}
{\em According to remark \ref{rk.lozengeannulus}, chains of lozenges
correspond to sequences of Birkhoff annuli,
every Birkhoff annulus sharing a common periodic orbit with the previous element
of the sequence, and also a periodic orbit with the next element in the sequence.
When the sequence closes up, it provides an immersion $f: T^2$ (or $K$) $\to M$, which
is called a} Birkhoff torus {\em (if the cyclic sequence contains an even number of
Birkhoff annuli),  or a} Birkhoff Klein bottle {\em (in the other case).}
\end{remark}

\section{Disjoint  pieces and transverse tori}
\label{sec:disj}
In \cite{bafe} section 7, we proved that if $P$ is
a periodic Seifert fibered piece of $M$ with
a pseudo-Anosov flow $\Phi$ then:
there is a connected, finite union $Z$ of elementary
Birkhoff annuli, which is
weakly embedded $-$ this means
that restricted to the union of the
interiors of the Birkhoff annuli it is embedded
and the periodic orbits are disjoint from the
interiors. 
In addition a small neighbourhood
$N(Z)$ is a {\em {representative}} for the Seifert fibered
piece $P$, that is, $N(Z)$ is isotopic to $P$.
We call such a $Z$ a {\em {spine}} for the
Seifert piece $P$.
In this section we prove several
important results concerning the relative
position of spines of distinct periodic
Seifert pieces (if there are such), and
also how the boundaries of such $N(Z)$
interact with the flow $\Phi$.

\begin{lemma}{}
Let $A_1, A_2$ be two elementary Birkhoff
annuli which lift to the
same lozenge $C$ in $\mi$ and so that the cores
of $A_1, A_2$ are freely homotopic.
Then $A_1$ is flow homotopic to
$A_2$ in the interior. That is, there is a homotopy $f_t$
from $A_1$ to $A_2$
so that for any $x$ in the interior of $A_1$,
\  $f_{[0,1]}(x)$
is contained in a flow line of $\Phi$. In addition
if $x$ is a point
where $A_1$ does not self
intersect,
then $f_t([0,1])(x)$
is a set of no self intersections,
of the homotopy.
\end{lemma}

\begin{proof}{}
Choose fixed lifts $\widetilde A_1, \widetilde A_2$
so that the interiors intersect exactly
the orbits in $C$.
Let $g$ in $\pi_1(M)$ so that  it generates
$Stab(\widetilde A_i)$.
The fact that a single $g$ generates both stabilizers
uses the condition on the cores of $A_1, A_2$.

Let $E$ be the interior of $\widetilde A_1$.
For any $p$ in $E$, there is
a unique real number $t(p)$ so that
$\wwp_{t(p)}(p)
= \eta_i(p)$ is a point in $\widetilde A_2$.
This map $t(p)$ is continuous and
clearly equivariant under $g$:
$\wwp_{t(g(p))} \ (g(p)) = g(\eta_i(p))$.
%In addition  this has a well defined
%continuous extension to
%$\widetilde A_1 \rightarrow \widetilde A_2$.
Since this is equivariant,
it projects to a map from the interior of $A_1$ to
the interior of
$A_2$. The linear homotopy
along the orbits in the required
homotopy. The homotopy is
an isotopy where $A_1$ is
embedded.
\end{proof}

\begin{proposition}{}{}
Let $\Phi$ be a pseudo-Anosov flow and
let   $\{ P_i, \ 1 \leq i \leq  n\}$
(where $n$ may be $0$) be the periodic
Seifert pieces of $\Phi$. Then we may choose the
spines $Z_i$ of $P_i$ so that they are pairwise
disjoint.
\label{disjspin}
\end{proposition}

\begin{proof}{}
The construction in \cite{bafe} shows that
the periodic orbits of each $Z_i$
are uniquely determined by $P_i$.
The Birkhoff annuli in $Z_i$ are not unique and
can be changed by flow homotopy.
We prove the proposition in several steps.

\vskip .1in
\noindent
I) For any $i \not = j$, a periodic orbit in $Z_i$ does
not intersect $Z_j$.

In this part there will be no need to make adjustments
to the Birkhoff annuli.
Suppose $\alpha$ is a closed orbit in $Z_i$ which
intersects $Z_j$ with $j, i$ distinct.
The first situation is that $\alpha$ intersects
a closed orbit $\beta$ in $Z_j$, in which
case $\alpha = \beta$.
Recall that in a periodic Seifert piece some power
$\alpha^{n_i}$ represents a regular fiber
in $Z_i$ and similarly some power $\beta^{n_j}$
represents a regular fiber in $Z_j$. But then
the regular fibers in $Z_i, Z_j$ have common
powers $-$ this is impossible for distinct
Seifert fibered pieces of $M$ \cite{He}.

The second situation is that $\alpha$ intersects
the interior of a Birkhoff annulus $A$ in
$Z_j$. Since $Z_j$ is a spine for the Seifert
piece $P_j$, there is an immersed Birkhoff torus
$T$ in $Z_j$ containing $A$.
%Since $Z_j$ is a spine for the Seifert
%piece $P_j$, there is an immersed Birkhoff torus
%$T$ in $Z_j$ containing $A$.
In addition choose $T$ to be $\pi_1$-injective.
This can be achieved by looking at a lift
$\widetilde T$ to $\mi$ and the sequence of lozenges
intersected by $\widetilde T$. If there is no
backtracking in the sequence of lozenges then
$T$ is $\pi_1$-injective. It is easy to choose one such
$T$ with no backtracking.

Fix a lift $\widetilde A$ of $A$ contained in
a lift $\widetilde T$ of $T$ and let $\widetilde
\alpha$ be a lift of $\alpha$ intersecting $\widetilde A$.
Since $T$ is incompressible, $\widetilde T$ is a properly
embedded plane in $\mi$.
The topological plane $\widetilde T$ is contained
(except for the lifts of the periodic orbits)
in a bi-infinite chain of lozenges ${\mathcal C}$.
%It is a properly embedded plane in $\mi$ and
Any orbit in the interior of one of
the lozenges in ${\mathcal C}$ intersects
$\widetilde T$ exactly once.

Since $\alpha$ corresponds to a closed curve
in $P_i$ and $Z_j$ is isotopic into
$P_j$, then $\alpha$ can be homotoped to
be disjoint from $Z_j$. Lift this homotopy
to $\widetilde M$ from $\widetilde \alpha$ to
a bi-infinite curve in $\widetilde M$ disjoint
from $\widetilde T$.
Recall that $\widetilde \alpha$ intersects
$\widetilde T$ in a single point.
This implies that
a whole ray of $\widetilde \alpha$ has
to move cross $\widetilde T$ by the homotopy. Hence this ray
is at bounded distance from $\widetilde T$. As $\alpha$ is
compact this implies that $\alpha$ is freely homotopic
into $T$. Let $g$ be the covering translation
which is a generator of $Stab(\widetilde \alpha)$.
Then with the correct choices we can assume
that $g$ is in $\pi_1(T)$, which we assume
is contained in $Stab(\widetilde T)$.

This is now a contradiction because $g$
leaves invariant the bi-infinite chain of lozenges
${\mathcal C}$. In addition $g(\widetilde \alpha)
= \widetilde \alpha$ so $g$ leaves
invariant the lozenge $B$ of ${\mathcal C}$
containing $\widetilde \alpha$. But then it would
leave invariant the pair of corners of $B$,
contradiction to leaving invariang
$\widetilde \alpha$.
We conclude that this cannot happen.
This finishes part I).

\vskip .1in
\noindent
II) Suppose that for some $i \not = j$ there are
Birkhoff annuli $A \subset Z_i, B \subset Z_j$
so that $A \cap B \not = \emptyset$.

Notice that the intersections are in the interior
by part I). Recall also that the interiors of
the Birkhoff annuli are embedded. By a small
perturbation put the collection $\{ Z_k \}$
in general position with respect to itself.
Let $\delta$ be a component of $A \cap B$.

Suppose first that $\delta$ is not null homotopic
in $A$. Since $A$ is $\pi_1$-injective, then
the same is true for $\delta$ in $M$ and $\delta$
in $B$. Then $\delta$ is homotopic in
$A$ to a power of a boundary of $A$, which
itself has a common power with the regular fiber
of $P_i$. This implies that the fibers in
$P_i, P_j$ have common powers, contradiction
as in part I).

It follows that $\delta$ is null homotopic
in $A$ and hence bounds a disc $D$ in $A$.
Notice that $\delta$ is embedded as both
$A$ and $B$ have embedded interiors.
We proceed to eliminate such intersections
by induction.
We assume that $\delta$ is innermost
in $D$: the interior of $D$ does not intersect
any $Z_k, k \not = i$ (switch $j$ if necessary).
In addition $\delta$ also bounds a disc
$D'$ in $B$ whose interior is disjoint
from $D$ $-$ by choice of $D$.
Hence $D \cup D'$ is an embedded
sphere which bounds a ball $B$ in $M$ $-$ because
$M$ is irreducible.
We can use this ball to isotope $Z_j$ to
replace a neighborhood of $D'$ in $B$
by a disc close to $D$ and disjoint from $D$,
eliminating the intersection $\delta$ and
possibly others. Induction eliminates
all intersections so we can assume that all
$\{ Z_k \}$ are disjoint (for a more detailed explanation of
this kind of argument, see \cite[section $7$]{Ba3}).

Notice that the modifications in $\{ Z_k \}$ in
part II) were achieved by isotopies.
This finishes the proof of the proposition.
\end{proof}

Recall that we are assuming the manifold $M$ to be orientable, so that
we can use \cite[Theorem F]{bafe}; however the following Lemma holds
in the general case, hence we temporally drop the orientability hypothesis.

\begin{lemma}{(local transversality)}{}
Let $V$ be an immersed Birkhoff torus or
Birkhoff Klein bottle with no backtracking
$-$ this means that for any lift of $V$ to
$\mi$, the sequence of lozenges associated
to it has no backtracking.
Let $\widetilde V$ be a fixed lift of $V$ to
$\widetilde M$ and let $\widetilde \alpha$
a lift to $\widetilde V$ of a closed orbit
$\alpha$ in $V$.
%If $V$ is a Klein bottle and $\alpha$
%is orientation reversing in $V$ (its
%neighborhood is a Moebius band) then
%we assume that $V$ is two sided in $M$.
There are well defined lozenges $B_1, B_2$ in
$\mi$ which contain a neighborhood of
$\widetilde \alpha$ in $\widetilde V$
(with $\widetilde \alpha$ removed): \
$\widetilde \alpha$ is a corner of both
$B_1$ and $B_2$.
%Let $E$ be a small neighborhood of $\alpha$
%in $V$.
%Let $\alpha'$ be a loop in $V$ which is a generator
%of $\pi_1(E)$. Let $g$ a covering translation
%associated to $\alpha'$ so that also $g(\widetilde \alpha)
%= \widetilde \alpha$.
If $B_1, B_2$ are adjacent  lozenges
%and $g$ preserves
%the pair $B_1, B_2$ then $V$ can
then $V$ can be homotoped to be transverse
to $\Phi$ near $\alpha$.
Conversely if
$V$ can be homotoped to be transverse
to $\Phi$ near $\alpha$
then $B_1$ and $B_2$ are adjacent.
This is independent of the lift $\widetilde V$
of $V$ and of $\widetilde \alpha$.
\label{localtrans}
\end{lemma}

\begin{proof}{}
Formally we are considering a map $f: T^2$ (or $K$) $\rightarrow
M$ so that the image is the union $V$ of (immersed) Birkhoff
annuli. The homotopy is a homotopy of the map $f$
and it may peel off pieces of $V$ which are glued
together. This occurs for instance if the orbit
$\alpha$ is traversed more than once in $V$, the
image of $f$.
An example of this is a Birkhoff annulus that wraps
around its boundary a number of times.
Another possibility is that many closed curves in
$T^2$ or $K$ may map to $\alpha$ and we are only
modifying the map near one of these curves.

Let $S$ be the domain of $f$ which is
either the torus $T^2$ or the Klein bottle
$K$.
%Choose a lift $\widetilde f:
%\widetilde S \rightarrow \mi$ which has image
%$\widetilde V$.
There is a simple closed curve $\beta$ in $S$
and a small neighborhood $E$ of
$\beta$ in $S$ so that $f(E)$ is also
the projection of a small neighborhood
of $\widetilde \alpha$ in $\widetilde V$ to $M$.
Notice that $E$ may be an annulus or Mobius band.
The statement ``$V$ can be homotoped to be
transverse to $\Phi$ in a neighborhood of $\alpha$"
really means that $f|E$ can be homotoped so that
its final image $f'|E$ is transverse to $\Phi$.
We will abuse terminology and keep referring
to this as ``$V$ can be homotoped ...".

%in $V$.
%Let $\alpha'$ be a loop in $V$ which is a generator
%of $\pi_1(E)$.
Let $g$ a covering translation
associated to $f(\beta)$. It follows that
$g(\widetilde \alpha) = \widetilde \alpha$.
In addition since $g$ is associated to a loop coming
from $S$ (and not just a loop in $V$), then
$g$ preserves $\widetilde V$ and more to the point
here $g$ preserves
the pair $B_1, B_2$.
It may be that $g$ switches $B_1, B_2$, for example
if $\beta$ is one sided in a Mobius band.
This is crucial here: if we took $g$ associated to $\alpha$
for instance, then $g(\widetilde \alpha) = \widetilde \alpha$,
but $g$ could scramble the lozenges with corner
$\widetilde \alpha$ in an unexpected manner and
one could not guarantee that $B_1, B_2$ would be
preserved by $g$.
We also choose $\beta$ so that $f(\beta) = \alpha$.

%First we explain why we consider $\alpha'$ and not just $\alpha$. This
%is because a Birkhoff annulus may wrap
%around its boundary more than once: in that
%case $\alpha'$ is freely homotopic to a
%power of $\alpha$. If however
%$E$ is a Mobius band then $\alpha'$ is
%freely homotopic to $\alpha$.
%Since we want to homotope
%$V$ near $\alpha$, the curve $\alpha'$ is
%the one that determines the first return map.

Suppose first $B_1, B_2$ are adjacent and wlog assume they
are adjacent along a half leaf $Z$ of $\widetilde W^u(\alpha)$.
The crucial fact here is that since $g$ preserves
the pair $B_1, B_2$ then $g$ leaves $Z$ invariant.
%see fig. \ref{adja}.
%
%\blankfig{adja}{0.1}{Birkhoff annuli in adjacent lozenges.}
%
Let $U$ be a neighborhood of $\alpha$ in $M$. Choose it
so the intersection with $f(E)$
is either an annulus or Mobius band (in general only
immersed). Using the image $\pi(Z)$ of the half leaf $Z$
we can homotope the power of $\alpha$ corresponding
to $g$ (that is, corresponding  to $f(\beta)$
as a parametrized loop) away from
$\alpha$
so that its image in $\pi(Z)$ is transverse
to the flow and closes up.
%$\widetilde W^u(\widetilde \alpha)$
In the universal cover $g$
preserves
the set $B_1, B_2$ then the pushed curve from $\widetilde \alpha$
returns
to the same sector of $\mi - \widetilde W^s(\widetilde \alpha)$
and this curve can be closed up when mapped to $M$.
Once that is done we can also homotope
a neighborhood of $\alpha$ in $V$ as well to be
transverse to $\Phi$.

In the most general situation that neighborhood of $V$ could
be an annulus which is one sided in $M$, then
the push away from $\alpha$ could not close up.
In our situation it may be that this annulus goes around say twice
over $\alpha$ and going once around $\alpha$ sends
the lozenges $B_1, B_2$ to other lozenges. But going
around twice over $\alpha$ (corresponding to $g$)
returns $B_1 \cup B_2$ to itself.
If the neighborhood is a Mobius band, we want to
consider the core curve as it generates the fundamental
group of this neighborhood.
This finishes the proof of the first statement of the lemma.

\vskip .1in
Suppose now that $V$ can be homotoped to be transverse
to $\Phi$ is a neighborhood of $\alpha$.
We use the same setup as in the first part.
Let $U$ be a neighborhood of $\alpha$ in $M$
so that the pulled back neighborhood of $\beta$ in $S$
is either an annulus or Mobius band.
We assume that $V$ can be perturbed near $\alpha$ to $V'$ in $U$,
keeping it
fixed in $\partial U$, and to be transverse
to $\Phi$ in a neighborhood of $\alpha$.
Let $A$ be the the part of $V'$ which is the part of $V$
perturbed near $\alpha$.

Consider all prongs of $(\ws(\widetilde \alpha)
\cup \wu(\widetilde \alpha)) - \widetilde \alpha$.
By way of contradiction we
are assuming that the lozenges $B_1, B_2$ are not
adjacent.  Then there are at least 2 such prongs
as above separating $B_1$ from $B_2$ in
$\mi - \widetilde \alpha$ on either component
of $\mi - (B_1 \cup B_2 \cup \widetilde \alpha)$.
Let $\widetilde A$ be the lift of
$A$ near $\widetilde \alpha$.

We first show that $\widetilde A \cap \widetilde \alpha$
is empty. Suppose not and let $p$ in the intersection.
Since $A$ is transverse to $\Phi$ then $\widetilde A$
is transverse say to $\wlu$ so we follow the intersection
$\widetilde A \cap \widetilde W^u(p)$ from $p$. This projects
to a compact set in $A$, contained in the interior of $A$
as $\widetilde W^u(p)$ does not intersect $\partial \widetilde A$.
This is because $\partial \widetilde A$ is contained
in the union of the lozenges $B_1, B_2$ and they are disjoint
from any prong of $p$.
So the original curve in $\widetilde W^u(p)$
has to return to $\widetilde \alpha$ and looking
at this curve in $\widetilde W^u(\alpha)$ this transverse
curve has to intersect $\widetilde \alpha$ twice, which
makes it impossible to be transverse to the flow.

Since $\widetilde A$ cannot intersect $\widetilde \alpha$
and it has boundaries in $B_1$ and $B_2$ then
it has to intersect at least two prongs from $\widetilde
\alpha$, at least one stable and one unstable prong
in $\widetilde U$. Project to $M$. Then $A$ cannot
be transverse to the flow $\Phi$.
This is because in a stable prong of $\alpha$ the flow
is transverse to
$A$ in one direction and in an unstable prong
of $\alpha$ the flow is transverse to $A$ in the opposite direction.
This finishes the proof of lemma \ref{localtrans}.
\end{proof}

\begin{proposition}{(transverse torus)}{}
\label{pro:transversetorus}
Let $\Phi$ be a pseudo-Anosov flow in $M^3$.
Suppose that $P_i, P_j$ are periodic Seifert
fibered pieces which are adjacent and let
$T$ be a torus in the common boundary of
$P_i, P_j$.
Then we can choose $N(Z_i), N(Z_j)$
neighborhoods of $P_i, P_j$ so that the
components  $T_i, T_j$ of $\partial N(Z_i), \partial N(Z_j)$
isotopic to $T$ are the same set and
this set is transverse to $\Phi$.
\label{transtor}
\end{proposition}

\begin{proof}{}
By proposition \ref{disjspin} we may assume that
$N(Z_i), N(Z_j)$ are disjoint. Since $Z_i$ is a spine for $P_i$,
the torus $T$ is homotopic to a Birkhoff torus $T_1$ contained
in $Z_i$.
We assume that $T_1$ has no backtracking.
Quite possibly
$T_1$ is only an immersed torus,
for example there may be Birkhoff annuli in $Z_i$ which
are covered twice by $T_1$.
The torus $T_1$ lifts to a properly embedded plane
$\widetilde T_1$ which intersects a unique
bi-infinite  chain of lozenges
${\mathcal B}_1$. With appropriate choices we may
assume that $\pi_1(T) \cong {\bf Z}^2$ corresponds to
a subgroup $G$ of covering translations leaving
${\mathcal B}_1$ invariant.
The corners of the lozenges in ${\mathcal B}_1$
project to closed orbits in $T_1$. These have powers
which are freely homotopic to the regular fiber in
$P_i$ because $P_i$ is a periodic Seifert
piece. Similarly $P_j$ produces a
Birkhoff torus $T_2$ homotopic to $T$ with $T_2$
contained in $Z_j$ and a lift $\widetilde T_2$
contained in a bi-infinite chain of lozenges
${\mathcal B}_2$, which is also invariant under the same
$G$. The corners of the lozenges in ${\mathcal B}_2$
project to closed orbits of the flow with powers
freely homotopic to a regular fiber in $P_j$.
If these two collections of corners are the
same, they have the same isotropy group, which
would imply the fibers in $P_i, P_j$ have common
powers, impossible as seen before.

We conclude that ${\mathcal B}_1, {\mathcal B}_2$ are
distinct and both invariant under $G \cong {\bf Z}^2$.
This is an exceptional situation and proposition
5.5 of \cite{bafe} implies that both
chains of lozenges are contained in a scalloped
region and one of them (say ${\mathcal B}_1$)
is s-scalloped and
the other (${\mathcal B}_2$) is u-scalloped.
The lozenges in the s-scalloped region all intersect
a common stable leaf, call it $E$ and the corners
of these lozenges are in stable leaves in
the boundary of the scalloped region.

Let then $\alpha$ be a periodic orbit in $T_1$ with
lift $\widetilde \alpha$ to $\widetilde T_1$
and lozenges $B_1, B_2$ of ${\mathcal B}_1$ which
have corner $\widetilde \alpha$. Then $B_1,
B_2$ are adjacent along an unstable leaf.
Further if $f(\beta)$ is the curve as in the
proof of the
previous lemma, which is homotopic to a power
of $\alpha$, then $f(\beta)$ is in $T$ and so the
covering translation associated to $f(\beta)$
preserves $\widetilde T_1$ and hence also
$B_1, B_2$. By the previous lemma we can homotope
$T_1$ slightly near $\alpha$
to make it transverse
to the flow near $\alpha$. When lifting to
the universal cover, the corresponding lift
of the perturbed torus will not intersect $\widetilde \alpha$,
but will intersect all orbits in the half leaf
of $\widetilde W^u(\widetilde \alpha)$ which
is in the common boundary of $B_1$ and $B_2$.
Do this for all closed orbits of $\Phi$ in $T_1$.
Notice we are pushing $T_1$ along unstable leaves.
Consider now a lozenge $B_1$ in ${\mathcal B}_1$ and
$A$ the Birkhoff annulus contained in the closure
of $\pi(B_1)$ which is contained in $T_1$.
Both boundaries have been pushed away along
unstable leaves. The unstable leaves are on the
same side of the Birkhoff annulus $A$.
Therefore
one can also push in the same direction the remainder
of $T_1$ $-$ to make it disjoint from $Z_i$. This produces
a new torus $T'_1$ satisfying

\begin{itemize}

\item $T'_1$ is a contained in 
%component of the boundary
a small neighborhood of $Z_i$ and is transverse to the
flow $\Phi$,

\item
$T'_1$ is disjoint from every $Z_k$ (including $Z_i$),

\item
There is a fixed lift $\widetilde T'_1$ which is
invariant under $G$ and that it intersects
exactly the orbits in the scalloped region.

\end{itemize}

The much more subtle property to prove is the following:

\vskip .1in
\noindent
{\bf Claim} $-$ $T'_1$ can be chosen embedded.
\vskip .03in

To prove this claim we fix the representative
$N(Z_i)$ of $P_i$ and a Seifert fibration $\eta_i: P_i \rightarrow \Sigma_i$
so that $Z_i$ is a union of fibers: each
Birkhoff annulus $A$ of $Z_i$ is a union
of fibers and it is embedded in the interior.
Since $M$ is orientable, the orbifold $\Sigma_i$ is a surface
with a finite number of singular points, which are the projections by $\eta_i$ of
the vertical periodic orbits. Moreover, $\eta_i(Z_i)$ is a {\em fat graph,} i.e.
is a graph embedded in $\Sigma_i$ which is a deformation retract of $\Sigma_i$.
One can furthermore select $\eta_i$ so that the stable and unstable vertical annuli in
$N(Z_i)$ are Seifert saturated, ie. project to arcs in $\Sigma_i$ with one boundary
in $\Sigma_i$, the other being a vertex of $\eta_i(Z_i)$. Finally, one can assume that
the retraction $r: \Sigma_i \rightarrow \eta_i(Z_i)$ is constant along stable and unstable arcs,
mapping each of them on the vertex of $\eta_i(Z_i)$ lying in their boundary.

Since the Birkhoff annuli in $Z_i$ are transverse to the flow,
one can distinguish the two sides of every edge of $\eta_i(Z_i)$, one where
the flow is ``incoming", and the other ``outgoing". The stable arcs are contained
in the incoming side, whereas the unstable arcs are contained in the outgoing side.
It follows that
the set of boundary components of $\Sigma$
can be partitioned in two subsets so that for every
edge $e$ of $\eta_i(Z_i)$, the two sides of $e$ in $\Sigma$ lie
in different sets of this partition.

The immersed Birkhoff torus $T_1$ is a sequence of Birkhoff annuli
$A_1$, $A_2$, ... , $A_k$, $A_{k+1} = A_1$. It corresponds to a sequence $e_1$, ... , $e_k$, $e_{k+1} = e_1$ of edges
in $\eta_i(Z_i)$. As described above, since ${\mathcal B}_1$ is $s$-scalloped,
$T'_1$ is obtained by pushing every $A_i$ along the unstable annuli so that
$T'_i$ intersects no stable annulus. It follows that we always push on the ``outgoing" side.
Let $c_i$ be the unique segment in the outgoing boundary of $\Sigma_i$ whose image by the
retraction $r$ is $e_i$: it follows that the sequence of segments $c_1$, $c_2$, ... , $c_k$ describe
an outgoing component $C$ of $\partial\Sigma_i$. In other words, $\eta_i(T_1)$
is the retraction of a boundary component of $N(Z_i)$.

Hence, if we have $e_i = e_j$ for some $i < j$,
we have $e_{i+1} = e_{j+1}$, and so on, so that the sequence $e_1$, ... , $e_k$
is the repetition of a single loop in $\eta_i(Z_i)$. Then, $T_1$ is homotopic to
the boundary component $\eta_i^{-1}(C)$ of $N(Z_i)$ repeated at least twice.
But it would mean that the JSJ torus $T$ is homotopic to the JSJ torus $\eta_i^{-1}(C)$
repeated several times, which is a clear contradiction.

Therefore, $e_1$, ... , $e_k$ is a simple loop: $T_1$ can pass through a Birkhoff annulus
in $Z_i$ at most once. Then the homotopy from $T_1$ to $T'_1$
does the following: the interiors of the Birkhoff annuli are
homotoped to an embedded collection. The neighborhoods of
the periodic orbits also satisfy that.
We conclude that $T'_1$ can be chosen embedded.
This finishes the proof of the claim.
\vskip .1in

As $T'_1$ is embedded and homotopic to $T$ and $M$ is
irreducible, then $T'_1$ is in fact isotopic to $T$
\cite{He}.
The same is true for $T_2$ to produce $T'_2$ with
similar properties.
Notice that $\widetilde T'_1$
and $\widetilde T'_2$ intersect exactly the
same set of orbits in $\mi$. Hence their projections
$T'_1, T'_2$ to $M$ bound a closed region $F$ in $M$ with
boundary $T'_1 \cup T'_2$, homeomorphic to $T'_1 \times
[0,1]$ and so that the flow is a product in $F$.
We can then isotope $T'_1$ and $T'_2$ along flow
lines to collapse them together.

In this way we produce representatives $N(Z_i), N(Z_j)$
of $P_i, P_j$ respectively; with boundary components
$T'_1$, $T'_2$ isotopic to $T$ (they are the same set)
and transverse to the flow $\Phi$.
This finishes the proof of proposition \ref{transtor}.
\end{proof}

\begin{proposition}{(good position)}{}
Suppose that $\Phi$ is a totally periodic
pseudo-Anosov flow in a graph manifold $M$.
Let $P_i$ be the Seifert fibered spaces in
the torus decomposition of $M$.
Then there are spines $Z_i$ made up of Birkhoff
annuli
for $P_i$
and compact neighborhoods $N(Z_i)$ so that:

\begin{itemize}

\item
$N(Z_i)$ is isotopic to $P_i$,

\item
The union of $N(Z_i)$ is $M$ and the interiors of
$N(Z_i)$ are pairwise disjoint,

\item
Each $\partial N(Z_i)$ is a union of tori
in $M$ all of which are transverse to the flow
$\Phi$.
\end{itemize}

\label{goodposition}
\end{proposition}

\begin{proof}{}
If $P_i$ and $P_j$ are adjoining, the previous proposition
explains how to adjust the corresponding components
of $\partial N(Z_i)$ and $\partial N(Z_j)$ to satisfy
the 3 properties for that component without changing
any of the $\{ Z_k \}$ or the other components
of $\partial N(Z_i), \partial N(Z_j)$.
We can adjust these tori in boundary
of the collection $\{ N(Z_i) \}$ one by one.
This finishes the proof. This actually shows
that any component of $M - \cup \ Z_i$ is
homeomorphic to $T^2 \times [0,1]$.
\end{proof}

This proposition shows that given any boundary torus
of (the original) $N(Z_i)$ it can isotoped to
be transverse to $\Phi$.
Fix a component $\widetilde Z_i$ of the inverse
image of $Z_i$ in $\mi$ and let $B_1$ be a lozenge
with a corner $\widetilde \alpha$
in $\widetilde Z_i$ and so that
$B_1$ contains a lift $\widetilde A$
of an (open) Birkhoff
annulus $A$  in $Z_i$.
Let $\alpha$ be the projection of $\widetilde \alpha$
to $M$.
The proof of
proposition \ref{transtor} shows that
for each side of $A$ in $M$ there is a torus
which is a boundary component of a small neighborhood
$N(Z_i)$ and which contains an annulus very close
to $A$. Going to the next Birkhoff annulus on
each torus beyond $\alpha$, proposition \ref{transtor}
shows that the corresponding lozenge $B$ is adjacent
to $B_1$. Hence we account for the two lozenges
adjacent to $B_1$. This can be iterated.
This shows that for any corner $\widetilde \alpha$,
\textit{every} lozenge with corner  $\widetilde \alpha$
contains the lift of the interiof of a Birkhoff annulus
in $\partial N(Z_i)$.
Hence there are no more lozenges with a corner
in $\widetilde Z_i$. Hence the pruning step
done in section 7 of \cite{bafe} is inexistent:
the collection of lozenges which are connected
by a chain of lozenges to any corner in $\widetilde
Z_i$ is already associated to $N(Z_i)$.
This shows the assertion above.

\section{Itineraries}
\label{sec:itinerary}
In the previous section, we proved that $M$ admits a JSJ decomposition so that every Seifert piece $P_i$ is
a neighborhood $N(Z_i)$ of the spine $Z_i$ and whose boundary is a
union of  tori transverse to $\Phi$. Denote by $T_1$, ... , $T_k$ the
collection of all these tori:
for every $k$, there is a Seifert piece $P_i$ such that $\Phi$ points outward $P_i$ along $T_k$, and another piece $P_j$
such that $\Phi$ points inward $P_j$ along $T_k$ (observe that we may have $i=j$, and also $P_i$ and $P_j$ may have several
tori $T_k$ in common).
It follows from the description of $N(Z_i)$ that the only full orbits
of $\Phi$ contained in $N(Z_i)$ are closed orbits and they
are the vertical orbits of $N(Z_i)$.

From now on in this section, we fix one such JSJ decomposition of $M$
associated to the flow $\Phi$.
In this section when we consider a Birkhoff annulus without any further
specification we are referring to a Birkhoff annulus in one
of the fixed spines $Z_i$. In the same way a
\textit{Birkhoff band} is a lift to $\mi$ of the interior of one
of the fixed Birkhoff annuli.

Let $\gsg(T_k)$, $\gug(T_k)$ be the foliations induced on $T_k$ by $\ls$, $\lu$. It follows from the previous section (and also by the Poincar\'e-Hopf index theorem)
that $\gsg(T_k)$ and $\gug(T_k)$ are regular foliations, i.e. that the orbits of $\Phi$ intersecting $T_k$ are regular. Moreover, $\gsg(T_k)$ admits
closed leaves, which are the intersections between $T_k$ and the stables leaves of the vertical
periodic orbits contained in $N(P_j)$.
Observe that all the closed leaves of $\gsg(T_k)$ are obtained in this way:
it follows from the fact that $T_k$ is an approximation of an union of Birkhoff annuli.

Hence there is a cyclic order on the set of closed leaves of $\gsg(T_k)$, two successive closed leaves for this order are the boundary components of an annulus in $T_k$ which can be pushed
forward along the flow to a Birkhoff annulus contained in the spine $Z_j$. We call such a region of $T_k$ an \textit{elementary $\gsg$-annulus of $T_k$.}

Notice that an elementary annulus is an open subset of the respective torus $T_k$ $-$ the
boundary closed orbits are not part of the elementary annulus.

Similarly, there is a cyclic order on the set of closed leaves of $\gug(T_k)$ so that the region between two successive closed leaves (an  \textit{elementary $\gug(T_k)$-annulus}) is obtained by
pushing forward along $\Phi$ a Birkhoff annulus appearing in $Z_i$.
The regular foliations $\gsg(T_k)$ and $\gug(T_k)$ are transverse one to the other and
their closed leaves are not isotopic. Otherwise $P_i, P_j$ have Seifert fibers
with common powers, contradiction.
Hence none of these foliations
admits a Reeb component. It follows that leaves in an elementary $\gsg(T_k)$ or $\gug(T_k)$-annulus
spiral from one boundary to the other boundary so that the direction of ``spiralling'' is the
opposite at both sides. It also follows that the length of curves in one leaf of these foliations not intersecting a closed leaf of the other foliation is uniformly bounded from above. In other words:

\begin{lemma}
\label{le:boundabove}
There is a positive real number $L_0$ such that any path contained in a leaf of $\gug(T_k)$ (respectively $\gsg(T_k)$) and contained in an elementary $\gsg(T_k)$-annulus (respectively $\gug(T_k)$-annulus) has length $\leq L_0$. \hfill $\square$
\end{lemma}

\vskip .1in
\noindent
{\bf {The sets $\mathcal T$ and $\mathcal T_\sharp$}} $-$
Let $\mathcal T$ be the collection of all the lifts in $\widetilde{M}$ of the tori $T_k$. Every
element of $\mathcal T$ is a properly embedded
plane in $\mi$. We will also abuse notation and denote by $\mathcal T$ the union of the elements of $\mathcal T$.
Let $\Delta$ be the union of the lifts of the vertical orbits of $\Phi$.
Finally let $\mathcal T_{\sharp} = \mathcal T \cup \Delta$.
\vskip .08in

Observe that there exists a positive real number $\eta$ such that the $\eta/2$-neighborhoods of the
$T_k$ are pairwise disjoint. Therefore:
\begin{equation}
\label{eq:TT'}
\forall \wt, \wt'\in {\mathcal T}, \;\; \wt \neq \wt' \Rightarrow d(\wt, \wt') \geq \eta
\end{equation}

\noindent
Here $d(\wt,\wt')$ is the minimum distance between a point in $\wt$ and a point in $\wt'$.

What we have proved concerning the foliations $\gsg(T_k)$, $\gug(T_k)$ implies the following: for every $\wt \in \cT$, the restrictions
to $\wt$ of $\wls$ and $\wlu$ are foliations by lines, that we denote by $\widetilde{\gsg}(\wt)$,
$\widetilde{\mathcal G}^u(\wt)$. These foliations are both product, i.e. the leaf space of each of them is homeomorphic to
the real line. Moreover, every leaf of $\widetilde{\mathcal G}^s(\wt)$ intersects every leaf of $\widetilde{\mathcal G}^u(\wt)$ in one
and only one point. Therefore, we have a natural homeomorphism $\wt \approx {\mathcal H}^s(\wt) \times {\mathcal H}^u(\wt)$,
identifying every point with the pair of stable/unstable leaf containing it (here, ${\mathcal H}^{s,u}(\wt)$ denotes the leaf space of $\widetilde{\mathcal G}^{s,u}(\wt)$).

\vskip .05in
\noindent
{\bf {Bands and elementary bands}} $-$
Some leaves of $\widetilde{\mathcal G}^{s,u}(\wt)$ are lifts of closed leaves: we call them \textit{periodic leaves}. They cut $\wt$ in bands, called (stable or unstable) \textit{elementary bands},
which are lifts of elementary annuli (cf. fig. \ref{fig:bands}). Observe that the intersection between a stable elementary
band and an unstable elementary band is always non-trivial: such an intersection is called a \textit{square}. Finally,
any pair of leaves $(\ell_1, \ell_2)$ of the same foliation $\widetilde{\mathcal G}^s(\wt)$ or $\widetilde{\mathcal G}^u(\wt)$ bounds a region in $\wt$
that we will call a \textit{band} (elementary bands defined above is in particular a special type of band).
A priori bands and elementary bands can be open, closed or ``half open" subsets of
$\wt$.

\begin{remark}
\label{rk:orientons}
{\em We arbitrarily fix a transverse orientation of each foliation $\gsg(T_k)$, $\gug(T_k)$. It induces (in a $\pi_1(M)$-equivariant way) an orientation on each leaf space ${\mathcal H}^{s,u}(\wt)$. Since every leaf of $\widetilde{\mathcal G}^u(\wt)$ is
naturally identified with ${\mathcal H}^{s}(\wt)$, the orientation of ${\mathcal H}^{s}(\wt)$ induces an orientation on every leaf of $\widetilde{\mathcal G}^u(\wt)$.
When one describes successively the periodic leaves of $\widetilde{\mathcal G}^u(\wt)$, this orientation alternatively coincide and not with the orientation induced
by the direction of the flow.
This is because such leaves are lifts of closed curves isotopic to periodic orbits.}

\end{remark}

\begin{figure}
\centeredepsfbox{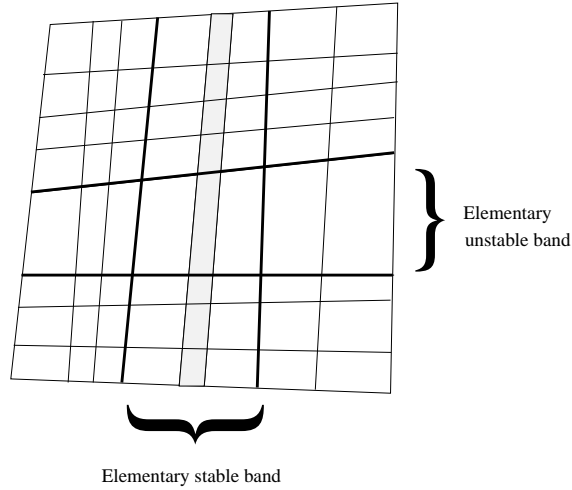}
\caption{
Bands and elementary bands. Nearly vertical lines are leaves of $\widetilde{\mathcal G}^s(\wt)$, and nearly
horizontal lines are leaves of $\widetilde{\mathcal G}^u(\wt)$. Thicker lines are periodic leaves.
The shaded region is a band which is not an elementary band.}
\label{fig:bands}
\end{figure}

For every $\tilde{x}$ in $\mi$, let $(\wt_1(\tilde{x}), ... , \wt_n(\tilde{x}), ...)$ be the list of the elements
of $\cT$ successively met by the positive $\wwp$-orbit of $\wx$ (including an initial $\wt$ if $\wx$ is contained in an element $\wt$ of $\cT$).
Observe that this sequence can be finite, even empty: it happens precisely when the positive orbit remains trapped in a connected component
of $\mi \setminus \cT$, i.e. the lift of a Seifert piece $P_i$. In this case, the projection of the orbit lies in the stable leaf of a vertical
periodic orbit $\theta$ of $P_i$. In other words, $\wx$ lies in $\wws(\alpha)$ where $\alpha$ is a lift of $\theta$. In that case,
we denote by $I^+(\wx)$ the sequence $(\wt_1(\tilde{x}), ... , \wt_n(\tilde{x}), \alpha, \alpha, ...)$, where $\wt_n(\wx)$ is the last element of $\cT$
intersecting the positive $\wwp$-orbit of $\tilde{x}$, and all the following terms are all equal to $\alpha$. We say then that $I^+(\wx)$ is \textit{finite.}
In the other case, i.e. when the sequence $(\wt_1(\tilde{x}), ... , \wt_n(\tilde{x}), ...)$ is infinite,
$I^+(\wx)$ will denote this infinite sequence.
In both situations, $I^+(\wx)$ is called the \emph{positive itinerary} of $\wx$.

Similarly, one can define the \textit{negative itinerary} $I^-(\wx)$ has the sequence of elements of $\cT$ successively crossed
by the negative orbit of $\wx$. Once more, such a sequence can be finite if $\wx$ lies in the unstable leaf of the lift of a periodic vertical orbit
$\beta$, in which case we repeatedly add this information at the end of the sequence. Actually, we consider $I^-(\wx)$ as a sequence indexed
by $0$, $-1$, $-2$, ...

\vskip .15in
\noindent
{\bf {Total itinerary and itinerary map}} $-$
The sequence
$I^-(\wx)$ together with $I^+(\wx)$ defines
a sequence indexed by $\mathbb Z$ called the \textit{total itinerary,} denoted by $I(\wx)$.
This defines a map $I: \mi \rightarrow \cT^{\bf Z}_\sharp$ called
the \textit{itinerary map}.

%In order to simplify the notations, we make the following convention: if $I^+(\wx)$ is finite, we complete it as a sequence indexed by the entire $\mathbb N$ by repeating the term $\alpha$ for all big terms. Applying a similar convention for $I^-(\wx)$, we obtain a notion of total itinerary indexed by the entire $\mathbb Z$ in all the cases.

\subsection{Characterization of orbits by their itineraries}

A very simple but crucial fact for the discussion here is the following:
if $\widetilde T$ is an element of $\cT$, then $\widetilde T$ is a properly embedded plane
transverse to $\wwp$. Hence it separates $\mi$ and intersects an arbitrary orbit
of $\wwp$ at most once.

\begin{lemma}
\label{le:memefeuille}
Let $\wt$ be an element of $\cT$. Let $\wx$, $\wy$ be two elements of $\wt$ such that $\wws(\wx) = \wws(\wy)$. Then
$I^+(\wx)=I^+(\wy)$.
\end{lemma}

\begin{proof}
Clearly:
$$\wt_1(\wx) = \wt = \wt_1(\wy)$$

For every integer $i$ such that $\wt_i(\wx)$ is well-defined, let $\wt^+_i$ be the connected component of $\mi \setminus \wt_i(\wx)$ not containing
$\wx$. One easily observes that if $i < j$, then $\wt_j^+ \subset \wt_i^+$.

We first consider the case where $I^+(\wx)$ is finite:
$$I^+(\wx) = (\wt_1(\tilde{x}), ... , \wt_n(\tilde{x}), \alpha, \alpha, ...)$$
Then, $\alpha \subset \wt_n^+$. Moreover, $\wy$ lies in $\wws(\wx) = \wws(\alpha)$, hence $I^+(\wy)$ is finite, and the $\wwp$-orbit
of $\wy$ accumulates on $\alpha$. It must therefore enter in $\wt^+_n$, hence intersects $\wt_n(\wx)$.
But for that purpose, it must enter in $\wt^+_{n-1}$, hence intersect $\wt_{n-1}(\wx)$. Inductively, we obtain that
$(\wt_1(\tilde{x}), ... , \wt_n(\tilde{x}))$ is a subsequence (in that order), of $I^+(\wy)$.

Since we can reverse the role of $\wx$ and $\wy$, we also prove in a similar way that $I^+(\wy)$ is a subsequence
of $I^+(\wx)$. By the remark above the equality $I^+(\wx) = I^+(\wy)$ follows.
\medskip

We consider now the other case; the case where $I^+(\wx)$ is infinite. Then, by what we have just proved above,
$I^+(\wy)$ is an infinite sequence too.
Recall that there is a positive real number $\eta$ bounding from below the distance between elements of $\cT$. In particular:
$$\forall i \in \mathbb N, \;\; d(\wt_i(\wx), \wt_{i+1}(\wx)) \geq \eta$$
Now, any length minimizing path between $\wt_i(\wx)$ and $\wt_{i+2}(\wx)$ must intersect
$\wt_{i+1}(\wx)$. It follows easily that:
$$\forall i \in \mathbb N, \;\; d(\wt_i(\wx), \wt_{i+2}(\wx)) \geq 2\eta$$
Inductively, one gets:
$$\forall i, p \in \mathbb N, \;\; d(\wt_i(\wx), \wt_{i+p}(\wx)) \geq p\eta$$

On the other hand, since $\widetilde y$ lies in $\wws(\wx)$, there is a positive real number $R$ such that:
$$\forall t>0, \;\; d(\wwp^t(\wx), \wwp^t(\wy)) \leq R$$

For every positive integer $n$, select $t \in {\mathbb R}^+$ such that $\wwp^t(\wx)$ lies in $\wt^+_{n+p}(\wx)$, where $p \geq 2R/\eta$. Then:
$$d(\wwp^t(\wx), \wt_n(\wx)) \geq p\eta \geq 2R$$

Since $d(\wwp^t(\wx), \wwp^t(\wy)) \leq R$, $\wwp^t(\wx)$ and $\wwp^t(\wy)$ lie on the same side of $\wt_n(\wx)$, i.e. $\wt_n^+$.
Hence $\wt_n(\wx)$ appears in the positive itinerary of $\wy$. Since $n$ is arbitrary, it follows that $I^+(\wx)$ is a subsequence
of $I^+(\wy)$, the order in the sequence being preserved.

Switching the roles of $\wx$ and $\wy$, we also get that $I^+(\wy)$ is a subsequence
of $I^+(\wx)$. Hence, the two itineraries must coincide.
\end{proof}

In order to prove the reverse statement, we need the following result:

\begin{proposition}
\label{pro:debutitinerary}
Let $\wt$, $\wt'$ be two elements of $\cT$, intersected successively by a $\wwp$-orbit, i.e. such that for some $\wz \in \wt$,
the first intersection of the forward orbit of $\wz$ with $\cT$ (after $\wt$) is some $\wwp^t(\wz) \in \wt'$,
$t > 0$. Then, the subset $\wA(\wt, \wt')$ comprised of elements of $\wt$ such that the positive itinerary starts by $(\wt, \wt', ...)$ is a stable elementary band;
more precisely, the stable elementary band of $\wt$ containing $\wz$.
\end{proposition}

\begin{proof}
The orbit of $\wz$ between the times $0$ and $t$
lies in a connected component of $\mi \setminus \cT$, hence projects into a Seifert piece $P_i$. Due to the previous section, this orbit
in $M$ starts in
a connected component of $\partial P_i$ (projection $T$ of $\wt$), intersects one of the Birkhoff annuli $A_0$ contained in the spine $Z_i$, and then crosses
the projection $T'$ of $\wt'$. In other words, there is a lift $\widetilde{A}_0$ of $A_0$ intersected by the orbit of $\wz$ and contained between $\wt$ and $\wt'$.
The boundary of $A_0$ is the union of two periodic orbits (maybe equal one to the other), and any element of $A_0$ has a negative orbit intersecting $T$, and a positive orbit intersecting $T'$. At the universal covering level, the boundary of $\widetilde{A}_0$ is the union of two distinct orbits $\alpha$ and $\beta$;
the $\wwp$-saturation of the Birkhoff band $\widetilde{A}_0$ intersects $\wt$ (respectively $\wt'$) along an elementary band $\widetilde{A}$ (respectively $\widetilde{A}'$).
Moreover, the boundary of $\widetilde{A}$ is the union of two leaves of $\widetilde{\mathcal G}^s(\wt)$. More precisely, one of these leaves is contained in the intersection
$\wt \cap S^1_\alpha$, and the other in the intersection $\wt \cap S^1_\beta$, where $S^1_\alpha$ is a component of $\wws(\alpha) \setminus \alpha$ and $S^1_\beta$ a component of $\wws(\beta) \setminus \beta$.

Similarly, $\widetilde{A}'$ is an elementary band in $\wt'$ bounded by two leaves of $\widetilde{\mathcal G}^u(\wt')$, which are contained in some
components $U^1_\alpha$, $U^1_\beta$ of $\wwu(\alpha) \setminus \alpha$, $\wwu(\beta) \setminus \beta$.

Clearly $\widetilde{A} \subset  \wA(\wt,\wt')$.

\textit{Claim $1$: the intersection $S^1_\alpha \cap \wt$ is connected.} If not, there would be a segment of orbit of $\wwp$ with extremities in $\wt$ but
not intersecting $\wt$. It would be in contradiction with the fact that $\wt$ disconnects $\mi$ and is transverse to $\wwp$.

Similarly, $S^1_\beta \cap \wt$ is connected (i.e. is reduced to a boundary component of the elementary band $\widetilde{A}$, and the intersections
$U^1_\alpha \cap \wt'$, $U^1_\beta \cap \wt'$ are connected, i.e. precisely the boundary components of $\widetilde{A}'$.

\textit{Key fact: by our definition of pseudo-Anosov flows, $\alpha$ and $\beta$ are not $1$-prong
orbits.} It follows that there is a component $S^2_\alpha$ of $\wws(\alpha) \setminus \alpha$ different from $S^1_\alpha$. We select this component to be the one just after $U^1_\alpha$, i.e. such that the following is true: the union $S(\alpha)=S^1_\alpha \cup \alpha \cup S^2_\alpha$ is a $2$-plane such that the connected component $C(\alpha)$ of $\mi \setminus S(\alpha)$ containing $\widetilde{A}_0$ does not intersect $\wws(\alpha)$. Similarly, we define a $2$-plane $S(\beta) = S^1_\beta \cup \beta \cup S^2_\beta$ contained in $\wws(\beta)$ such that the connected component $C(\beta)$ of $\mi \setminus S(\beta)$ containing $\widetilde{A}_0$ does not intersect $\wws(\beta)$ (see fig. \ref{fig:trapp}).

\begin{figure}
\centeredepsfbox{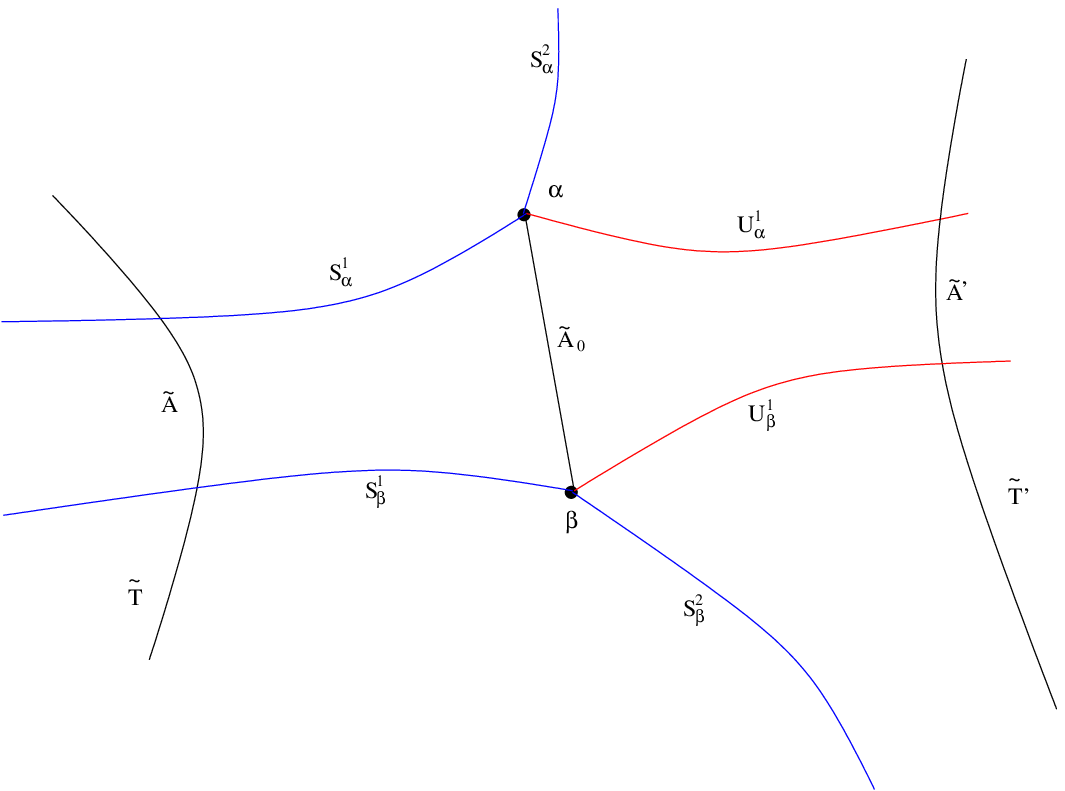}
\caption{
Isolating $\wt'$ and $\widetilde{A}_0$ by stable prongs.}
\label{fig:trapp}
\end{figure}

\textit{Claim $2$: $S^2_\alpha$ and $S^2_\beta$ are disjoint from $\wt'$:} indeed, the positive orbit of any point in $\wt'$ is trapped into the component of
$\mi \setminus \wt'$ disjoint from $\widetilde{A}_0$, hence cannot accumulate on $\alpha$ or $\beta$: $\wt'$ is disjoint from $\wws(\alpha) \cup \wws(\beta)$. The claim follows.
\medskip

\textit{Claim $3$: $S^2_\alpha$ and $S^2_\beta$ are disjoint from $\wt$:} Assume by contradiction that $S^2_\alpha$ intersects $\wt$.
In the same way we have proved that $S^1_\alpha \cap \wt$ is connected, one can prove that $S^2_\alpha \cap \wt$ is a single leaf of $\widetilde{\mathcal G}^s(\wt)$. It bounds together with
$S^1_\alpha \cap \wt$ a $\widetilde{\mathcal G}^s(\wt)$-band $\widetilde{B}$. Let $g$ be the indivisible element of
$\pi_1(M)$ corresponding to a generator of the fundamental group of
the projection of $\alpha$: it preserves $\widetilde \alpha$, $\wt$, hence also $S^1_\alpha \cap \wt$, $S^2_\alpha \cap \wt$ and the band $\widetilde{B}$. It follows
that the union of $\widetilde{B}$ with the regions in $S^1_\alpha$, $S^2_\alpha$ between $S^1_\alpha \cap \wt$, $S^2_\alpha \cap \wt$ and $\alpha$ projects in
$M$ as a torus contained in the Seifert piece $P_i$. This torus bounds a solid torus in $P_i$. Every orbit of $\phi$ entering in this solid torus from the projection of
$\widetilde{B}$ cannot exit from this solid torus, since it cannot further intersect the projection of $\widetilde{B}$, nor the projections of $S^1_\alpha$, $S^2_\alpha$.
It has to remain in the Seifert piece containing the projection of $\alpha$. It is a contradiction since the set of positive orbits trapped in $P_i$ has empty interior. Claim $3$ is proved.
\medskip

We now focus our attention to the region $C(\alpha) \cap C(\beta)$, whose boundary is the disjoint union $S(\alpha) \bigsqcup S(\beta)$. According to Claim $2$,
$\wt'$ is contained in $C(\alpha) \cap C(\beta)$. Now it follows from Claim $3$ that the intersection $C(\alpha) \cap C(\beta) \cap \wt$
is the elementary band $\widetilde{A}$.

Consider now the positive orbit of an element $p$ of $\wt \setminus \widetilde{A}$: if $p$ lies in $\partial\widetilde{A}$,
then this orbit accumulates on $\alpha$ or $\beta$ and therefore does not intersect $\wt'$. If not, then this orbit is disjoint from $\wws(\alpha) \cup \wws(\beta)$, hence never enters in $C(\alpha) \cap C(\beta)$. In particular, it never crosses $\wt'$.

It follows that $\wA(\wt,\wt') \subset \wA$ and so
proposition \ref{pro:debutitinerary}  is proved.
\end{proof}

We can now prove the converse of Lemma \ref{le:memefeuille}:

\begin{lemma}
\label{le:memefeuille2}
Let $\wt$ be an element of $\cT$. Let $\wx$, $\wy$ be two elements of $\wt$ such that $I^+(\wx)=I^+(\wy)$. Then $\wy$ and $\wx$
lie in the same leaf of $\widetilde{\mathcal G}^s(\wt)$.
\end{lemma}

\begin{proof}
Let $\wx_1$, $\wx_2$, ... and $\wy_1$, $\wy_2$, ... be the positive iterates of $\wx$, $\wy$ belonging in $\wt_1 := \wt_1(\wx) = \wt_1(\wy)$, $\wt_2 := \wt_2(\wx)  = \wt_2(\wy)$, ... .
According to Proposition~\ref{pro:debutitinerary}, for every positive integer $i$, the iterates $\wx_i$, $\wy_i$ lie in the same stable elementary band $\wA_i := \wA(\wt_i, \wt_{i+1}) \subset \wt_i$.

Consider first the case where the common itinerary $I^+(\wx) = I^+(\wy)$ is finite, of length $n+1$: $\wx_n$ and $\wy_n$ lie in the elementary band of $\wt_n$, and $I_k(\wx) = I_k(\wy) = \alpha$ for some periodic orbit $\alpha$, for all $k>n$. Hence,
th positive orbits of $\wx_n$, $\wy_n$ accumulate on $\alpha$. It follows from the arguments used in the proof of Claim $3$ of Proposition~\ref{pro:debutitinerary} that the
intersection $\wws(\alpha) \cap \wt_n$ is a single leaf of $\widetilde{\mathcal G}^s(\wt_n)$. Proposition~\ref{le:memefeuille} follows easily in this case.

We are left with the case where $I^+(\wx) = I^+(\wy)$ is infinite.
Let $\widetilde{\mathcal G}_i^s$, $\widetilde{\mathcal G}_i^u$ denote the restriction to $\widetilde{A}_i$ of $\wls$, $\wlu$. Observe that every leaf of $\widetilde{\mathcal G}_i^u$ intersects every leaf of $\widetilde{\mathcal G}_i^s$.
In particular, the $\widetilde{\mathcal G}_1^s$-leaf of $\wx$ intersects the $\widetilde{\mathcal G}_1^u$-leaf of $\wy$. Therefore, according to Lemma~\ref{le:memefeuille}, one can assume without loss of generality
that $\wx$ and $\wy$ lies in the same leaf of $\widetilde{\mathcal G}_1^u$. More precisely, there is a path $c: [a,b] \to \widetilde{A}_1$ contained in a leaf of $\widetilde{\mathcal G}_1^u$ and joining $\wx$ to $\wy$.

\textit{Assume by way of contradiction that $c$ is not a trivial path reduced to a point.}
If we push $c$ forward by the flow $\wwp$, one get a path $c_2: [a,b] \to \widetilde{A}_2$, contained in a leaf of $\widetilde{\mathcal G}_2^u$,
connecting $\wx_2$ to $\wy_2$. By induction, pushing along $\wwp$, we get a sequence of unstable paths $c_i: [a,b] \to \widetilde{A}_i$. Now, since
all these paths are obtained from $c$ by pushing along $\wwp$, the length of $c_i$ is arbitrarily long if $i$ is sufficiently big. This contradicts
Lemma~\ref{le:boundabove}.

This contradiction shows that $c$ is reduced to a point, i.e. that $\wx$ and $\wy$ lie in the same leaf of $\widetilde{\mathcal G}^s(\wt)$.
\end{proof}

Applying Lemmas~\ref{le:memefeuille} and \ref{le:memefeuille2} to the reversed flow one obtains:

\begin{proposition}
\label{pro:memefeuille2}
Let $\wt$ be an element of $\cT$. Let $\wx$, $\wy$ be two elements of $\wt$. Then $I^-(\wx)=I^-(\wy)$ if and only if $\wy$ and $\wx$
lie in the same leaf of $\widetilde{\gug}(\wt)$.\hfill $\square$
\end{proposition}

Itineraries are elements of ${\mathcal I} := \cT_\sharp^{\mathbb Z}$ where $\cT_\sharp$ is the disjoint union of $\cT$ with the set $\Delta$ of lifts of
vertical periodic orbits of $\Phi$. We define the \textit{shift map} $\sigma: {\mathcal I} \to {\mathcal I}$ which send any sequence $( \zeta_i )_{i \in \mathbb Z}$ to the sequence
$( \zeta_{i+1} )_{i \in \mathbb Z}$. Clearly, if $\wx$ and $\wy$ are two elements of $\mi$ lying on the same orbit of $\wwp$, then
$I(\wy)$ is the image of $I(\wx)$ under some iterate $\sigma^k$. Conversely:

\begin{corollary}
\label{cor:sameitinerary}
Let $\wx$, $\wy$ be two elements of $\mi$. Then $\wx$ and $\wy$ lie in the same orbit of $\wwp$ if and only if $I(\wy)=\sigma^k(I(\wx))$ for some $k \in \mathbb Z$.
\end{corollary}

\begin{proof}
Assume that $I(\wy)=\sigma^k(I(\wx))$ for some $k \in \mathbb Z$. 
Suppose first that the $\wwp$ orbit of $\wx$ does not intersect
$\cT$. Then $\wx$ is in $\Delta$ and similarly $\wy$ is also
in $\Delta$. The hypothesis immediately imply that $\wx, \wy$ are in the
same orbit of $\wwp$.
If $\wx$ intersects $\cT$, then after a shift if necessary,  we may assume that 
$T_1(\wx)$ is an element of $\cT$.
Then, by replacing $\wy$ by the element of its $\wwp$-orbit in the element $T_1(\wx) = T_{1+k}(\wy)$ of $\cT$, and $\wx$
by its iterate in $T_1(\wx)$, one can assume that $\wx$ and $\wy$ both lie in $T_1(\wx)$, and that $I(\wx)=I(\wy)$.
In particular, $I^+(\wx)=I^+(\wy)$ and $I^-(\wx) = I^-(\wy)$.
Then, according to Lemma~\ref{le:memefeuille} and proposition~\ref{pro:memefeuille2}, $\wx$ and $\wy$ have the same stable leaf and the same unstable leaf.
The corollary follows.
\end{proof}

In the same way one can prove that for any $\wx, \wy$ in $\mi$, then 
$\wx,\wy$ are in the same stable leaf of $\wwp$ if and only if the
positive iteneraries of $\wx,\wy$ are eventually equal up to a fixed shift.

We can now extend Proposition~\ref{pro:debutitinerary}:
\begin{proposition}
\label{pro:itinerary}
Let $\wt$, $\wt'$ be two elements of $\cT$. Then, the subset $\wA(\wt, \wt')$
of $\wt$ comprised of elements of $\wt$ whose positive orbits intersects $\wt',$ if non-empty, is a stable band.
Furthermore, let $\wx$ be an element of $\wA(\wt, \wt')$; its positive itinerary has the form $(\wt_1 := \wt, \wt_2, ... , \wt_n, ....)$ where $\wt_n = \wt'$. Then,
for any other element $\wy$ of $\wA(\wt, \wt'),$ the first $n$-terms of $I^+(\wy)$ are also $(\wt_1 := \wt, \wt_2, ... , \wt_n)$. The elements $\wt_2$, ... , $\wt_{n-1}$ are precisely the elements of $\cT$ that
separate $\wt$ from $\wt'$ in $\mi$.
\end{proposition}

\begin{proof}
Assume that $\wA(\wt, \wt')$ is non-empty, and let $\wx$ be an element of $\wA(\wt, \wt')$. Its positive itinerary contains $\wt'$, hence has the form $(\wt_1 := \wt, \wt_2, ... , \wt_n, ....)$ described in the statement.
As we have observed in the proof of Lemma~\ref{le:memefeuille}, we have $\wt^+_j \subset \wt^+_i$ for every $1\leq i < j \leq n$, hence every $\wt_i$ for $1 < i < n$ disconnects $\wt$ from
$\wt'$. On the other hand, every element of $\cT$ disconnecting $\wt$ and $\wt'$ must appear in the positive itinerary of elements of $\wA(\wt, \wt')$. It follows easily that the first $n$-terms
of the positive itinerary of elements of $\wA(\wt, \wt')$ coincide as stated in Proposition~\ref{pro:itinerary}. Moreover, if an element of $\wt$ has a positive itinerary of the form $(\wt_1 := \wt, \wt_2, ... , \wt_n, ....)$,
it obviously belongs to $\wA(\wt, \wt')$.

The only remaining point to check is that $\wA(\wt, \wt')$ is a stable band. But this follows immediatly from Lemma~\ref{le:memefeuille}.

It is very useful to give more information here: if $n = 2$ then $\wA(\wt,\wt')$ is a stable
{\underline {elementary}} band which projects to an open annulus in $T$. If $n > 2$ then $\wA(\wt,\wt')$ is
a stable band which is not elementary. For simplicity we describe the case $n = 3$. With the notation
above, then $\wA(\wt_2,\wt_3)$ is a stable elementary band as proved in
Proposition~\ref{pro:debutitinerary}. This band in $\wt_2$ intersects the unstable elementary
bands of $\wt_2$ in open squares. The one which has points flowing back to $\wt_1$ has
boundary made up of two stable sides $a_1, a_2$ which are contained in leaves of
$\widetilde{\gsg}(\wt_2)$, and unstable sides $b_1, b_2$ contained
in leaves of $\widetilde{\gug}(\wt_2)$.
Flowing back to $\wt = \wt_1$ (in this case) produces
$\wA(\wt,\wt')$. The arcs $b_1, b_2$ flow back towards two vertical periodic orbits,
without ever reaching them.
The arcs $a_1, a_2$ flow back to two full stable leaves $a'_1, a'_2$ of
$\widetilde{\gsg}(\wt)$.
Then $\wA(\wt,\wt')$ is the stable band with boundary $a'_1, a'_2$. This is not an elementary stable band.
In fact more is true: this band is strictly contained in a unique elementary
band and does not share a boundary component with this elementary band.
Finally this stable band projects injectively to $T$ $-$ unlike
what happens for elementary bands.

If $n > 3$ this process can be iterated.
Using the notation above the stable band bounded by $a'_1, a'_2$ intersects
the unstable elementary bands in their lifted torus in squares.
When flowing backwards, the same behavior
described above occurs.
\end{proof}

\begin{define}
\label{def:n}
Let $\wt$, $\wt'$ be two elements of $\cT$. We define the \textit{(signed) distance} $n(\wt, \wt')$ as follows:

-- if $\wt = \wt'$, then $n(\wt, \wt')=0$,

-- if $\wA(\wt, \wt')$ is non empty, then $n(\wt, \wt')$ is the integer $n$ such that for every element
$\wx$ of $\wA(\wt, \wt')$, we have $\wt_{n+1}(\wx) = \wt'$; and $n(\wt', \wt) = -n(\wt, \wt')$,

-- if $\wA(\wt, \wt')$ and $\wA(\wt', \wt)$ are both empty, then $n(\wt, \wt') = n(\wt', \wt) = \infty$.
\end{define}

\subsection{Behavior of the first return map}
\label{sub:stablestable}
Let us consider once more two successive elements $\wt$, $\wt'$ of $\cT$, i.e. such that $n(\wt, \wt')=1$. Recall that there is a stable elementary band $\wA := \wA(\wt, \wt') \subset \wt$ bounded by two stable leaves $l_1$, $l_2$,
and an unstable elementary band $\wA'  \subset \wt'$ bounded by two unstable leaves $l'_1$, $l'_2$, such that orbit of $\wwp$  intersecting
$\wt$ and $\wt'$ intersects them in precisely $\wA$, $\wA'$, respectively. The union of all these orbits is a region of $\mi$ bounded by (see fig. \ref{fig:trapp}):

-- $\wA$ and $\wA'$;

-- two stable bands $S^1_\alpha$, $S^1_\beta$ where $\alpha$, $\beta$ are lifts of periodic orbits,

-- two unstable bands $U^1_\alpha$, $U^1_\beta$.

\noindent
More specifically here
$S^1_{\alpha}$ denotes the unique component of
$\wws(\alpha) - (\alpha \cup \wt)$ whose closure intersects
both $\alpha$ and $\wt$. Its boundary is the union of $\alpha$ and a stable leaf
in $\wt$.
We call such a $3$-dimensional region the \textit{block defined by $\wt$, $\wt'$;} the orbits $\alpha$, $\beta$ are the \textit{corners} of the block.

\vskip .15in
\noindent
{\bf {The map $f_{\wt,\wt'}$}} $-$
We have a well-defined map $f_{\wt,\wt'}: \wA \to \wA'$, mapping every point to the intersection between its positive $\wwp$-orbit and $\wt'$. As long as there is no ambiguity on $\wt$ and $\wt'$, we will denote $f_{\wt,\wt'}$ by $f$.

\vskip .09in
Clearly, if two elements of $\wA$ lie on the same leaf of $\widetilde{\mathcal G}^s(\wt)$, then $f(\wx)$ and $f(\wx')$ lie on the same leaf of $\widetilde{\mathcal G}^s(\wt')$.
In other words, $f$ induces a map

$$f^s  :=  f_{\wt,\wt'}^s: \ (l_1, l_2) \ \to \ {\mathcal H}^s(\wt'),$$

\noindent
where $(l_1, l_2)$
is the open segment of the leaf space $\widetilde{\mathcal H}^s(\wt)$ delimited by
the boundary leaves $l_1$ and $l_2$ of $\wA$.

Since every leaf of $\widetilde{\mathcal G}^s(\wt')$ intersects every leaf of $\widetilde{\mathcal G}^u(\wt')$, it follows that $f^s$ is
{\underline {surjective}}:
if $v$ is a leaf of $\widetilde{\mathcal G}^s(\wt')$ then the intersection
property implies that $v$ intersects $l'_1$.
Then $v$ intersects $\wA'$ so $v$ is in the image of $f$.
Moreover, the intersection between every leaf of $\widetilde{\mathcal G}^s(\wt')$ and $\wA'$ is connected, hence $f^s$ is one-to-one.

Now assume that as in fig. \ref{fig:tilt} the flow along the ``periodic orbit" $\alpha$ is going up. Then, $\beta$ is going down. It follows that the map $f$ has the following behavior: points in $\wA$ close to $l_1$ are sent by $f$ in the top direction of $\wA'$, meaning that the closer to $l_1$ is the point $\wx$, the upper is the image $f(\wx)$. Indeed, the closer to $l_1$ is $\wx$, the longest is the period of time the positive orbit of $\wx$ will follow the vertical direction of $\alpha$. On the other hand, when $\wx$ is going near to $l_2$, the image $f(\wx)$ will be going the closer to $l'_2$, and in the bottom direction.

We already know that stable leaves in $\wt'$ cross the two sides $l'_1$ and $l'_2$ of $\wA'$, hence can be drawn as in the picture in a nearly horizontal way (since we have drawn $l'_{1,2}$ as vertical lines). Therefore, stable leaves in $\wA$, which are the pull-back by $f$ of stable leaves in $\wt$, are as depicted in fig. \ref{fig:tilt}: if we describe such a leaf $s$ from the bottom to the top, the image $f(s)$ will go from the left ($l'_1$) to the right ($l'_2$).
It follows that what is at the right (respectively, at the left) of $s$ is mapped under $f$ below (respectively, above) $f(s)$.

Observe that if we reverse the direction of the flow on $\alpha$ (and hence also on $\beta$), when we would have the opposite behavior: $f$ would map what is on the right of $s$ above $f(s)$.

\begin{figure}
\centeredepsfbox{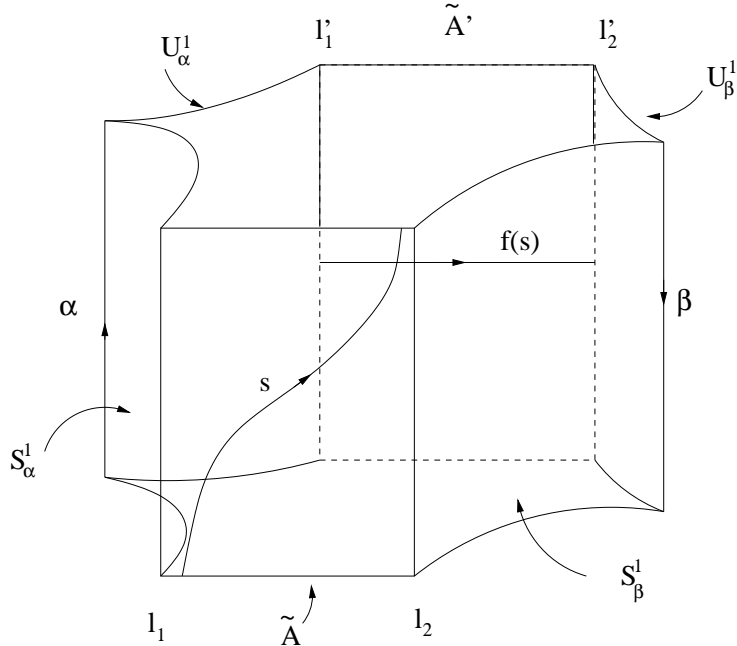}
\caption{
Pushing along the flow a stable leaf in an elementary band.}
\label{fig:tilt}
\end{figure}

Now recall that we have arbitrarily fixed an orientation of ${\mathcal H}^s(\wt)$ and ${\mathcal H}^s(\wt')$
(Remark~\ref{rk:orientons}). It is equivalent to prescribe a total order $\prec$ on ${\mathcal H}^s(\wt)$ and
${\mathcal H}^s(\wt')$. Assume e.g. that the positive
orientation on ${\mathcal H}^s(\wt)$ - which is a space of roughly vertical lines in $\wA$ - is from the left to the right; and assume that the orientation on ${\mathcal H}^s(\wt')$ is from the top to the bottom. Then, if $\alpha$ is oriented from the bottom to the top (as in the figure), the map $f^s$ preserves the orientation, whereas if $\alpha$ as the inverse orientation, $f^s$ reverses the orientations.

\vskip .18in
\noindent
{\bf {The map $f_{\wt,\wt'}$ when $n(\wt,\wt') > 1$}} $-$
Now if $\wt$, $\wt'$ are elements of $\cT$ with $n(\wt, \wt') = n > 1$, we still have a map
$f_{\wt, \wt'}: \wA(\wt, \wt') \to \wA'(\wt, \wt')$ where $\wA(\wt, \wt')$ is a stable band
in $\wt$ and $\wA'(\wt, \wt')$ an unstable band in $\wt'$. More precisely, the positive itinerary of orbits starting from $\wt$ and crossing $\wt'$ starts by $\wt_1 := \wt$, $\wt_2$, ... , $\wt_n := \wt$. Then $f_{\wt, \wt'}$ is the composition of all the $f_{\wt_i, \wt'_{i+1}}$.
The fact that the domain of
$f_{\wt,\wt'}$ is a stable band was proved in the end of the proof
of proposition \ref{pro:itinerary}. To get that the image is an unstable
band notice that flowing
backwards instead of forwards shows this fact (as the image is the domain
of the inverse map).

Exactly as in the case that $n(\wt,\wt') = 1$, it follows that the
map $f_{\wt,\wt'}$
also induces a
{\bf {surjective}} map $f^s_{\wt, \wt'}$ from a segment of ${\mathcal H}^s(\wt)$ onto
the entire ${\mathcal H}^s(\wt')$. This map can preserve the orientation or not; this property
depends on the orientation of the corners of the Birkhoff annuli successively crossed.

\subsection{Realization of itineraries}
\label{su.realization}
%We define an oriented graph $\widetilde{\mathfrak G}$ as follows:

%\begin{itemize}
%  \item vertices are connected components of $\mi \setminus \cT$,
%  \item edges are elements of $\cT$,
%  \item the initial vertex and the final vertex of an oriented edge $\wt$ are the components $\widetilde{P}_i$, $\widetilde{P}_j$ such that $\wt$ is a  boundary of $\widetilde{P}_i$ and of $\widetilde{P}_j$, so that any small orbit of $\wwp$ crossing $\wt$ exits from $\widetilde{P}_i$ and enters in $\widetilde{P}_j$.
%\end{itemize}

%We can define in a similar way an oriented graph $\mathfrak G$ whose edges are the tori $T_k$ and whose vertices are the Seifert pieces $P_i$. The definition
%of this graph is similar to the definition of $\widetilde{\mathfrak G}$; details are left to the reader.

%There is an obvious map $q: \widetilde{\mathfrak G} \to \mathfrak G$ induced by the universal covering map $p: \mi \to M$.

%!!!!!!!!!!!!!!!!!!!!!!!!!!!!!!!!!!!!!!!!!!!!!!!!!!!!!!!!

We define an oriented graph $\widetilde{\mathfrak G}$ as follows:

\begin{itemize}
  \item vertices are elements of $\cT$,
  \item edges are Birkhoff bands,
  \item the initial vertex and the final vertex of an oriented edge $E$ are the elements $\wt$, $\wt'$ of $\cT$ such that there are orbits of $\wwp$ intersecting $\wt$ at a point $\wx$, then crossing
$E$, and crossing afterwards $\wt'$ at a point $\wx'$. We require furthermore that $E$ is the unique Birkhoff band intersected by the orbit between $\wx$ and $\wx'$; in other words, that $n(\wt, \wt')=1$.
\end{itemize}

%We can define in a similar way an oriented graph $\mathfrak G$ whose edges are the Birkhoff
%annuli contained in on the spines $Z_i$ and whose vertices are
%the boundary tori $T_k$ of the Seifert pieces. The definition
%of this graph is similar to the definition of $\widetilde{\mathfrak G}$; details are left to the reader.

We add to $\widetilde{\mathfrak G}$ some vertices: the set $\Delta$ of lifts of vertical periodic orbits. These new vertices will not be connected one to the other, but only
to vertices of $\widetilde{\mathfrak G}$: we add an edge oriented from $\wt$ to $\alpha$ (respectively from $\alpha$ to $\wt$) if some element $\wx$ of $\wt$ has a positive (respectively negative) orbit accumulating on $\alpha$
without intersecting any element of $\cT$. The result is the \textit{augmented graph} $\widetilde{\mathfrak G}^\sharp$.
%Once more, this definition can be reproduced in $M$, leading to the augmented graph ${\mathfrak G}^\sharp$.

%There are obvious maps $q: \widetilde{\mathfrak G} \to \mathfrak G$ and $q^\sharp: \widetilde{\mathfrak G}^\sharp \to {\mathfrak G}^\sharp$ induced by the universal covering map $p: \mi \to M$.

\begin{lemma}
\label{le:connected}
The graphs $\widetilde{\mathfrak G}$ and $\widetilde{\mathfrak G}^\sharp$ are (weakly) connected.
\end{lemma}

\begin{proof}
Recall that an oriented graph is weakly connected if the underlying non-oriented graph is connected, i.e. if
any pair of vertices can be connected by a sequence of edges.
%Since $\mathfrak G$, ${\mathfrak G}^\sharp$ are the images of $\widetilde{\mathfrak G}$, $\widetilde{\mathfrak G}^\sharp$, we just have to prove that $\widetilde{\mathfrak G}$ and $\widetilde{\mathfrak G}^\sharp$ are connected.
It is quite obvious that $\widetilde{\mathfrak G}^\sharp$
is weakly connected as soon as $\widetilde{\mathfrak G}$ is weakly connected.

For every element $\wt$ of $\cT$ let $U_1(T)$ the open subset of $\mi$ comprising elements attained by an orbit of $\wwp$ intersecting $\wt$, and let $W_1(\wt)$
be the set of elements of $\cT$ which can be joined to $\wt$ by an orbit of $\wwp$. Define then inductively:

\begin{eqnarray*}
    U_{i+1}(\wt) & = & \bigcup_{\wt' \in W_i(\wt)} U_1(\wt')\\
    W_{i+1}(\wt) & = & \bigcup_{\wt' \in W_i(\wt)} W_1(\wt')
\end{eqnarray*}

We obtain an increasing sequence of domains $U_i(\wt)$ whose union $U_\infty(\wt)$ is an open subset of $\mi$; more precisely,
of $\mi \setminus \Delta$ (recall that $\Delta$ is the union of lifts of vertical periodic orbits).
In a less formal way, one can define $U_\infty(\wt)$ as the set of elements of $\mi \setminus
\Delta$ which are attainable from $\wt$  through
concatenations of segments of orbits of $\wwp$ and paths in elements of $\cT$.

The domains $U_\infty(\wt)$ where $\wt$ are elements of
$\cT$ are either pairwise disjoint or equal.
Since any orbit of $\wwp$ which is not the lift
of a vertical periodic orbit
intersects one element of $\cT$, it follows that
the union of these domais is the entire set $\mi \setminus \Delta$. This domain is
connected, hence $\mi \setminus \Delta = U_\infty(\wt)$
for every $\wt$ in $\cT$. In particular, for every $\wt$, $\wt'$ in $\cT$, there is an integer $i$ such that $\wt'$ lies in $W_i(\wt)$.

Now observe that the sequence of elements of $\cT$ successively crossed by an orbit of $\wwp$ defines a path in $\widetilde{\mathfrak G}$.
It follows that $\widetilde{\mathfrak G}$ is weakly connected, as required.
\end{proof}

As we have observed in the previous proof, every oriented path in $\widetilde{\mathfrak G}^\sharp$ defines naturally an element of ${\mathcal I} = \cT_\sharp^{\mathbb Z}$.
Recall that $\cT_\sharp = \cT \cup \Delta$.

Let ${\mathcal I}_0 \subset \mathcal I$
be the subset of ${\mathcal I}$ comprising sequences $(\zeta_i)_{i \in \mathbb Z}$ corresponding to
oriented paths in $\widetilde{\mathfrak G}^\sharp$
satisfying the following additional property:

\begin{center}
    \textit{if $\zeta_i$ is an element of $\Delta$, then $\zeta_j = \zeta_i$ for either all $j \leq i$, or all $j \geq i$.}
\end{center}

\begin{proposition}
\label{prop:imageI}
The image of the itinerary map $I: \mi \to \cT_\sharp^{\mathbb Z}$ is precisely ${\mathcal I}_0$.
\end{proposition}

\begin{proof}
The fact that the image of $I$ is contained in ${\mathcal I}_0$ is quite obvious since if an itinerary
$I(\wx) = (\zeta_i)_{i \in \mathbb Z}$ has a term $\zeta_i$ equal to $\alpha \in \Delta$,
then either $\wx \in \wwu(\alpha)$, or $\wx \in \wws(\alpha)$. In the first case, $\zeta_j = \alpha$ for all $j \leq i$, whereas in the second case
$\zeta_j = \alpha$ for all $j \geq i$.

Let now $(\zeta_i)_{i \in \mathbb Z}$ be an element of ${\mathcal I}_0$.
If every $\zeta_i$ lies in $\Delta$, then they are all equal (since there is no edge in
$\widetilde{\mathfrak G}^\sharp$ connecting two different elements of $\Delta$): every $\zeta_i$ is
equal to $\alpha \in \Delta$. Then the sequence is the itinerary
of any element of $\alpha$.

Assume now that some $\zeta_i$ is an element $\wt$ of $\cT$, suppose this is
$\zeta_1$.
Consider positive integers $n$: as long as $\zeta_n$ is an element of $\cT$ (and not of $\Delta$), then the signed distance $n(\zeta_1, \zeta_n)$ (in the sense of Definition~\ref{def:n}) is $+n$. It follows that $\wA(\zeta_1, \zeta_n)$ is a (non-empty!)
stable band (usually not elementary). At the leaf space level
($\widetilde{\mathcal H}^s(\wt)$), the projection of $\wA(\zeta_1, \zeta_n)$
is an open segment $J(\zeta_1, \zeta_n)$ in $\widetilde{\mathcal H}^s(\wt) \approx \mathbb R$.

Assume first the case where the positive itinerary is finite: there is an integer
$n>0$ such that $\zeta_i \in \cT$ for all $i\leq n$, and such that $\zeta_{n+1}$ is an element $\alpha$
of $\Delta$. Then, there is one (and only one) stable leaf $s_0$ of $\zeta_n$ whose elements has positive itinerary $(\alpha, \alpha, ...)$. Since $f^s_{\zeta_1,\zeta_n}$ is surjective,
there is a stable leaf $s$ in $J(\zeta_1, \zeta_n)$ whose image by $f^s_{\zeta_1,\zeta_n}$
is $s_0$. Then, the positive itinerary of elements of $s$ is, as required, $(\zeta_i)_{i\geq 1}$.

Consider now the other case, that is, the case where every $\zeta_i$ ($i>0$) is an element of $\cT$.
Then, the segments \ $(J(\zeta_1, \zeta_i)_{i\geq 1})$ \
form a decreasing (for the inclusion) sequence of intervals in $\widetilde{\mathcal H}^s(\wt) \approx \mathbb R$.
In fact more is true. The explanation at the end of the proof of proposition
\ref{pro:itinerary} shows that when $n$ increases by one, then {\underline {both}}
endpoints of $J(\zeta_1,\zeta_n)$ change. This follows from the fact in that
explanation that the band inside the elementary band did not share a boundary
component with the elementary band.
Given this fact, it follows that the intersection of
the $J(\zeta_1,\zeta_n)$ is non-empty. Every $\wx$ in $\zeta_1$ whose projection lies in this intersection will admit as positive itinerary $(\zeta_i)_{i\geq 1}$.

In both situations, we have a non-empty stable band $\wA((\zeta_i)_{i\geq 1})$ comprising elements
of $\cT$ with positive itinerary $(\zeta_i)_{i\geq 1}$. Observe that according to
Lemma~\ref{le:memefeuille2}, $\wA((\zeta_i)_{i\geq 1})$ is a single stable leaf $-$ that
is, it is a degenerate stable band.

By applying this argument to the reversed flow, one gets that the set of elements of $\zeta_0$ whose negative itinerary coincide with
$(\zeta_i)_{i \leq 0}$ is an \textit{unstable} leaf. Since in the plane $\zeta_1$ every unstable leaf intersects every stable leaf, we obtain
that $\zeta_1$ contains exactly one element whose itinerary is precisely $(\zeta_i)_{i \in \mathbb Z}$.
\end{proof}

\section{Topological and isotopic equivalence}

Let $\Phi$, $\Psi$ be two totally periodic pseudo-Anosov flows, and let $\{ Z_i(\Phi)\} $, $\{ Z_i(\Psi) \} $
be respective (chosen) spine collections.

\subsection{Topological equivalence}
\label{sub.theoremD}
In this section we show how to deduce Theorem D from Theorem D'. 

\vskip .08in
\noindent
{\bf {Proof of theorem D}} 
 $-$ First suppose that $\Phi$
and $\Psi$ are topologically equivalent. Let $f$ be a self homeormorphism
of $M$ realizing this equivalence.
Consider the two JSJ decompositions $\{ f(N(Z_i(\Phi)) \}$ and $\{ N(Z_j(\Psi)) \}$ 
of $M$. By uniqueness these are isotopic, so for any $i$ there is
a $j = j(i)$ so that $f(N(Z_i(\Phi))$ is isotopic to $N(Z_{j(i)})$.
The map $f$ induces an isomorphism between the subgroups
$\pi_1(P_i)$ and $\pi_1(P_j)$. 
In \cite{bafe} the spine $Z_i(\Phi)$ was constructed as follows:
consider all elements $g$ in $\pi_1(P_i)$ which act freely in the orbit
space of $\wwp$. Each such $g$ leaves invariant a unique  bi-infinite chain
of lozenges ${\mathcal C}$, the corners of which project
to periodic orbits of $\Phi$. The union of these periodic orbits over all
$g$ gives the vertical periodic orbits of $P_i$, which hence depend
only on the flow $\Phi$ and the piece $P_i$.
The Birkhoff annuli in $Z_i$ are elementary and they are
determined up to $\Phi$-flow isotopy. These facts are equivariant under the map $f$.
This means that 
$f$ sends the vertical orbits of $Z_i(\Phi)$ to the vertical 
orbits of $Z_j(\Psi)$ and takes the Birkhoff annuli in $Z_i(\Phi)$
to Birkhoff annuli in $M$ which are $\Psi$-flow isotopic to the
Birkhoff annuli in $Z_j(\Psi)$. We can postcompose the map $f$ with 
the $\Psi$-flow isotopy
so that the conjugacy $f'$ takes $Z_i(\Phi)$ to $Z_j(\Psi)$.
%the collection 
%$\{ Z(\Phi) \}$ onto $Z(\Psi)$, preserving 
The conjugacy $f'$ preserves the orientation of the vertical orbits.
This finishes the proof of this direction.

\vskip .1in
Conversely, assume that up to a homeomorphism, we have the equality $\{ Z_i(\Phi) \}
  = \{ Z_j(\Psi) \}$,
and that the two flows define the same orientation on the vertical orbits. 
Up to reindexing the collection $\{ Z_j(\Psi) \}$ we can assume that for all
$i$, $Z_i(\Phi) = Z_i(\Psi)$.
Since the JSJ decomposition 
of $M$ is unique up to isotopy \cite{Ja-Sh,Jo}, it follows that
$N(Z_i(\Phi))$ is isotopic to 
$N(Z_i(\Psi))$ for all $i$. 
A torus $T$ boundary of $N(Z_i(\Phi))$ and $N(Z_j(\Phi))$ is isotopic to a corresponding
boundary torus $T'$ between $N(Z_i(\Psi))$ and $N(Z_j(\Psi))$.
We can then change the flow $\Psi$ by an isotopy so that $T' = T$.
Hence we can assume that $N(Z_i(\Phi)) = N(Z_i(\Psi))$.
We simplify the notations by setting $Z_i = Z_i(\Phi) = Z_i(\Psi)$ and $N(Z_i) = N(Z_i(\Phi)) = N(Z_i(\Psi))$.

For each component $T_k$ of $\partial N(Z_i)$, let ${\mathcal G}^{s,u}_\Phi(T_k)$ and ${\mathcal G}^{s,u}_\Psi(T_k)$ be the foliations on $T_k$ induced by the stable/unstable foliations
of respectively $\Phi$, $\Psi$.
Of course, a priori there is no reason for ${\mathcal G}^{s,u}_\Phi(T_k)$ and ${\mathcal G}^{s,u}_\Psi(T_k)$ 
to be equal. 
However we prove the following crucial properties:

\vskip .08in
\noindent
{\bf {Claim}} $-$ The two foliations 
${\mathcal G}^{s,u}_\Phi(T_k)$ and ${\mathcal G}^{s,u}_\Psi(T_k)$ 
have the same number of closed leaves, and these leaves
are all vertical.  It follows that these closed curves are isotopic.
In addition one can assume that the elementary ${\mathcal G}^{s,u}_\Phi$-annuli are exactly
the elementary ${\mathcal G}^{s,u}_\Psi$-annuli.

Consider the component $W_k$ of $N(Z_i) - Z_i$ containing
$T_k$ in its boundary. Then $W_k$ is homeomorphic
to $T^2 \times [0,1)$, where $T_k = T^2 \times \{ 0 \}$ is a boundary 
component of $N(Z_i)$ and is therefore entering or exiting $N(Z_i)$.
Suppose without loss of generality that $T_k$ is an outgoing component.
By the description of the flow in $N(Z_i)$ every point in $W_k$ flows backward to
intersect $Z_i$ or be asymptotic to a vertical orbit in $Z_i$.
Let $\alpha$ be a vertical orbit of $Z_i$. 
%Since $N(Z_i)$ could a priori
%be chosen arbitrarily near $Z_i$ and 
Since $T_k$ is outgoing it follows
that unstable leaf of 
$\alpha$ intersects $W_k$ and consequently this unstable leaf intersects 
$T_k$,
and in a closed leaf. It follows that for any such $\alpha$ there is
a closed leaf of ${\mathcal G}^s_{\Phi}(T_k)$ and a closed leaf of
${\mathcal G}^s_{\Psi}(T_k)$. These closed curves in $T_k$ have
powers which are freely homotopic to powers of the regular fiber
in $P_i$ and hence they have powers which are freely homotopic
to each other. 
Since both are simple closed curves in $T_k$ and $N(Z_i)$ is Seifert fibered,
it now follows that 
these closed leaves are isotopic in $T_k$.
This proves that that 
${\mathcal G}^{s,u}_\Phi(T_k)$ and ${\mathcal G}^{s,u}_\Psi(T_k)$ have the
same number of closed leaves and they are all isotopic.

The manifold $M$ is obtained from the collection of Seifert pieces
$\{ P_i \}$ by glueings along the collection of tori $\{ T_k \}$.
Isotopic glueing maps of the $\{ T_k \}$ generate the same manifold.
Hence we can change the glueing maps and have a homeomorphism of
$M$ which sends the closed leaves of ${\mathcal G}^{s,u}_\Phi(T_k)$
to the closed leaves of ${\mathcal G}^{s,u}_\Psi(T_k)$.
%The homeomorphism
%It follows that, up to isotopy, one can assume that they have the same closed leaves. 

%We will now show that 
%one can assume that the elementary ${\mathcal G}^{s,u}_\Phi$-annuli are exactly
%the elementary ${\mathcal G}^{s,u}_\Psi$-annuli.
%More precisely: every ${\mathcal G}^{s,u}_\Phi$-annulus (respectively ${\mathcal G}^{s,u}_\Psi$-annulus)
%is obtained by pushing in $N(Z_i)$ along the orbits of $\Phi$ (respectively $\Psi$)
%of an elementary annulus in $Z_i$. One can select the homeomorphism above so that
%the ${\mathcal G}^{s,u}_\Phi$-annulus corresponding to a given elementary Birkhoff annulus
%is equal to the ${\mathcal G}^{s,u}_\Psi$-annulus corresponding to the {\em same} Birkhoff annulus.

%Let now $W_k$ be the component of $N(Z_i) - Z_i$ containing $T_k$: 
%it is homeomorphic to the product $T^2 \times [0, 1[$. Assume without less of generality that $T_k$
%is a outgoing component. 
The unstable vertical annuli for $\Phi$ (respectively for $\Psi$)
connect the vertical periodic orbits to the closed leaves in $T_k$. Now the point is that
the vertical annuli for $\Psi$ may not be isotopic rel boundary to the vertical annuli for $\Phi$:
they may wrap around $T_k$, intersecting several times the vertical annuli for $\Phi$.
However, there is an orientation preserving homeomorphism which is the identity outside 
$\overline W_k$, 
inducing a Dehn twist around $T_k$ in the mapping class group of $M$, which maps every $\Phi$-unstable annulus
to the corresponding $\Psi$-unstable annulus. 
The completion of $W_k$ is a manifold which is a quotient of $T^2 \times [0,1]$.
The only identifications are in ``vertical" orbits in $T^2 \times \{ 1 \}$. The maps
above induce a homeomorphism of this quotient of $T^2 \times [0,1]$ which is the identity in
the boundary. It follows that this is isotopic to a ``Dehn twist" in the boundary $T^2 \times \{ 0 \}$.
We leave the details to the reader.

Now the conjugate of $\Phi$ by this homeomorphism has precisely the same spine decomposition $N(Z_i)$, the same orientation on vertical periodic orbits, and the same stable/vertical annuli
in each $N(Z_i)$. According to Theorem D', to be proved in the next section, this flow
is isotopically equivalent to $\Psi$. This finishes the proof of Theorem D.

\subsection{Isotopic equivalence}
\label{sub.theoremD'}
In this section we prove Theorem D'. Suppose first that there is a homeomorphism
$f$ isotopic to the identity,
realizing a conjugacy between $\Phi$ and $\Psi$ and preserving
direction of vertical orbits. 
As in the previous section we may assume there is such $f$ so that
$f(Z_i(\Phi)) = Z_i(\varphi)$ (after reindexing). In addition
if $\alpha$ is a vertical orbit in $Z_i(\Phi) = Z_i(\Psi)$,
then since $f$ is a topological conjugacy, it sends the stable/unstable
leaf of $\alpha$ to the stable/unstable leaf of $f(\alpha)$.
As $f(N(Z_i(\Phi)) = N(Z_i(\Psi))$, it now follows that
$f$ maps the collection of vertical annuli of $\Phi$ in $N(Z_i(\Phi))$ to
the collection of vertical annuli of $\Psi$ in $N(Z_i(\Psi))$. This proves one 
direction of theorem D'.

\vskip .05in
Conversely we assume as in the previous section that for every
$i$ we have $Z_i(\Phi) = Z_i(\Psi) = Z_i$ (after reindexing), 
that $\Phi$, $\Psi$ define the same spine
decomposition of $M$ in periodic Seifert pieces $N(Z_i)$, and that they induce the same orientation on the vertical periodic orbits in each $Z_i$.
We furthermore assume that in
each $N(Z_i)$ they have exactly the same stable/unstable vertical annuli.
%As already mentioned in the introduction, Theorem D' will be proved if
%we show that under these hypothesis, $\Phi$ and $\Psi$ are isotopically equivalent.

Recall the following several objects we introduced in section~\ref{sec:itinerary}:

\begin{itemize}
  \item the set $\cT$ of lifts of components of $N(Z_i)$,
  \item the set $\Delta$ of lifts of vertical periodic orbits,
  \item elementary bands for $\wwp$,
\end{itemize}

The hypothesis imply that 
these objects coincide precisely with the similar objects associated to $\Psi$.

Observe that in every Seifert piece $N(Z_i)$, the ``blocks" connecting one incoming boundary torus to an outgoing torus are
delimited by vertical stable/unstable annuli, hence are exactly the same for the two flows. It follows that
the graph $\widetilde{\mathfrak G}$ and the augmented graph $\widetilde{\mathfrak G}^\sharp$ are precisely the same
for the two flows. Indeed: let $\wt$, $\wt'$ are two vertices of the graph $\widetilde{\mathfrak G}$ for $\Phi$, connected by one edge, then
there is a lift $\widetilde{N(Z_i)}$ such that $\wt$ and $\wt'$ are boundary component of $\widetilde{N(Z_i)}$, and an orbit of $\wwp$  inside $\widetilde{N(Z_i)}$
joining a point $\tilde{x}$ in $\wt$ to a point $\tilde{y}$ in $\wt'$. This orbit is trapped in the lift
$\widetilde{U}$ of a block as described in section 
4, delimited by two stable bands, two unstable bands, and two elementary bands, one in $\wt$, the other in $\wt'$ 
(we refer once again to figure \ref{fig:trapp}). It follows that
the orbit of $\widetilde{\Psi}$ starting at $\tilde{x}$ is trapped in the same lifted block $\widetilde{U}$ until
it reachs the same elementary band as the one attained by the $\wwp$-orbit, so that we have also $n(\wt, \wt') = +1$ from the view point of $\Psi$.

Hence, Proposition~\ref{prop:imageI} implies that any itinerary
of $\widetilde{\Psi}$ is realized by $\wwp$, and vice-versa.

Since the flows $\Phi$ and $\Psi$ have precisely the same blocks, with the same orientation
on the periodic corner orbits, we can apply section~\ref{sub:stablestable} and its conclusion. More precisely: for every lifted torus $\wt$, let ${\mathcal H}^s_\Psi(\wt)$ denote the leaf space of
the restriction to $\wt$ of the stable foliation of $\widetilde{\Psi}$. If $\wt'$ is another lifted torus such that $n(\wt, \wt') > 0$, we can define a map $g^s_{\wt, \wt'}$,
analogous to $f^s_{\wt, \wt'}$, from an interval of ${\mathcal H}^s_\Psi(\wt)$ onto ${\mathcal H}^s_\Psi(\wt')$.
This map is obtained using the flow $\widetilde{\Psi}$.

We also need to define the transverse orientation to the
foliations ${\mathcal G}^s_\Psi(T_k)$
and ${\mathcal G}^u_\Psi(T_k)$. These are the stable and unstable foliations
induced by the flow $\Psi$ in the JSJ tori $\{ T_k \}$.
As observed in the previous section, the closed leaves of these foliations are the same
as the closed leaves of the foliations
${\mathcal G}^s_\Phi(T_k)$
and ${\mathcal G}^u_\Phi(T_k)$ induced by $\Phi$ on $T_k$.
Choose the transverse orientation of (say)
${\mathcal G}^s_\Psi(T_k)$ to have it agree with the transverse orientation of
${\mathcal G}^s_\Phi(T_k)$ across the closed leaves.

By hypothesis, the flow directions of corresponding vertical periodic orbits of
$\Phi$ and $\Psi$ in any given Seifert piece $P_i$ agree.
This implies that the holonomy of the foliations
${\mathcal G}^s_\Phi(T_k)$ and
${\mathcal G}^s_\Psi(T_k)$ along the closed leaves agrees with each other,
that is, they are either both contracting or both repelling.

The important conclusion is that
with this choice of orientations and the above remark,
$g^s_{\wt, \wt'}$ is orientation preserving
if and only if $f^s_{\wt, \wt'}$ is orientation preserving.

Let $\oo_\Phi$, $\oo_\Psi$ denote the orbit spaces of $\wwp$, $\widetilde{\Psi}$, respectively. According to Corollary~\ref{cor:sameitinerary} applied to $\Phi$ and to
$\Psi$ as well, there is a natural bijection $\varphi: \oo_\Phi \to \oo_\Psi$: the one mapping an orbit of $\wwp$ to the unique orbit of $\widetilde{\Psi}$ admitting
the same itinerary up to the shift map. Observe that the map $\varphi$ is naturally $\pi_1(M)$-equivariant.

\begin{lemma}
\label{le:continuous}
The map $\varphi: \oo_\Phi \to \oo_\Psi$ is a homeomorphism.
\end{lemma}

\begin{proof}
The only remaining point to prove is the continuity of $\varphi$ (the continuity of the inverse map $\varphi^{-1}$ is obtained by reversing the arguments below).
We already know that two orbits lie in the same stable leaf if and only if their itineraries, up to a shift, coincide after some time. Hence, $\varphi$
maps the foliation $\oo_\Phi^s$ onto the foliation $\oo_\Psi^s$, and similarly, $\varphi$ maps
$\oo_\Phi^u$ onto $\oo_\Psi^u$.
%Hence $\varphi$ induces maps
%$\varphi^s: {\mathcal H}^s \to {\mathcal H}_\Psi^s$ and $\varphi^u: {\mathcal H}^u \to {\mathcal H}_\Psi^u$, where ${\mathcal H}^{s,u}$ and
%${\mathcal H}_\Psi^{s,u}$ are the respective leaf spaces.

Observe that the projection of $\Delta$ by the map
$\Theta: \mi \rightarrow \oo_\Phi$
is a closed discrete subset of $\oo_\Phi$ that we denote by
$\Delta_\Phi$, and the image of $\Delta$ by
$\varphi$ is a closed discrete subset $\Delta_\Psi$
of $\oo_\Psi$ corresponding to the lifts of vertical periodic orbits of $\Psi$.

We first show the continuity of $\varphi$ on $\oo_\Phi \setminus \Delta_\Phi$. Let $\theta$
be an element of $\oo_\Phi \setminus \Delta_\Phi$. This is an orbit
of $\wwp$ which
crosses some element $\wt$ of $\cT$ at a point $\wx$.
The restriction to $\wt$ of the projection map
$\Theta: \mi \to \oo_\Phi$ is injective as remarked before.
Let $\cp_\Phi$ be the projection
of $\wt$ to $\oo_\Phi$; it is an open $2$-plane contained in $\oo_\Phi \setminus \Delta$. Elements of $\cp_\Phi$ are
characterized by the property that their itinerary (which is well-defined up to the shift map) contains
$\wt$. It follows that the image of $\cp_\Phi$ by $\varphi$ is the projection $\cp_\Psi$ of $\wt$ in $\oo_\Psi \setminus \Delta_\Psi$. We will show that
the restriction of $\varphi$ to $\cp_\Phi$ is continuous, which will prove as required that
$\varphi$ is continuous on $\oo_\Phi \setminus \Delta_\Phi$.
The restrictions $\cp^s_\Phi$, $\cp^u_\Phi$
of $\oo_\Phi^s$ and $\oo_\Phi^u$ to $\cp_\Phi$ are the projections of $\widetilde{\mathcal G}^s_\Psi(\wt)$,
$\widetilde{\mathcal G}^u_\Psi(\wt)$. They are regular foliations; more precisely, they are product foliations,
transverse to each other. Every leaf of $\cp^s_\Phi$ intersects every
leaf of $\cp^u_\Phi$ in one and only one point. In summary, $\cp_\Phi$ is homeomorphic
to ${\mathcal H}^s_{\cp_\Phi} \times {\mathcal H}^u_{\cp_\Phi} \approx {\mathbb R} \times {\mathbb R}$, where
${\mathcal H}^{s,u}_{\cp_\Phi}$ denotes the leaf space of $\oo_\Phi^{s,u}$ 
restricted to $\cp_\Phi$ respectively.

Similarly, $\cp_\Psi$ is homeomorphic to ${\mathcal H}^s_{\cp_\Psi} \times {\mathcal H}^u_{\cp_\Psi}$. Furthermore, since $\varphi$ maps $\oo_\Phi^s$, $\oo_\Phi^u$ onto $\oo_\Psi^s$, $\oo_\Psi^u$, the restriction of
$\varphi$ to $\cp \approx {\mathcal H}^s_{\cp_\Phi} \times {\mathcal H}^u_{\cp_\Phi}$ has the form:
$$(S,U) \to (\varphi^s(S), \varphi^u(U))$$
where $S$, $U$ denote leaves of $\cp^s_\Phi$, $\cp^u_\Phi$ and
$\varphi^s(S), \varphi^u(U)$ are the induced images in the
leaf space level.

Each leaf space ${\mathcal H}^s_{\cp_\Phi}$, ${\mathcal H}^s_{\cp_\Psi}$ admits a subdivision in \textit{elementary segments} $J_\Phi(\wt, \wt')$ (respectively $J_\Psi(\wt, \wt')$), which are the projections of elementary bands $A(\wt, \wt')$ where $n(\wt, \wt')=1$. Since the elementary bands in
$\wt$ for $\wwp$ and $\widetilde{\Psi}$ are exactly the same, the map $\varphi^s$ preserves the order between elementary segments.
This uses the choice of transverse  orientations for ${\mathcal G}^s_\Phi(T)$ and
${\mathcal G}^s_\Psi(T)$.

Let $S$, $S'$ be arbitrary elements of ${\mathcal H}^s_{\cp_\Phi}$, such that $S \prec S'$ for the order defined on ${\mathcal H}^s_{\cp_\Phi} \approx {\mathcal H}^s_\Phi(\wt)$ in Remark \ref{rk:orientons}.
Then, if $S$, $S'$ lie in different elementary bands, it follows from what we have just seen
that $\varphi^s(S) \prec \varphi^s(S')$, since $\varphi^s$ preserves the order between the elementary intervals $J_\Phi(\wt, \wt')$ and $J_\Psi(\wt, \wt')$.

Assume now that $S$, $S'$ lie in the same elementary band. Since $S \neq S'$, there is an element
$\wt'$ of $\cT$ such that $f^s_{\wt, \wt'}(S)$ and $f^s_{\wt, \wt'}(S')$
either lie in different elementary segments of ${\mathcal H}^s_\Phi(\wt')$
or lie in the closure of the same elementary segment.
The second case is equivalent to $S$ or $S'$ being in the stable manifold
of a lift of a vertical periodic orbit.
We will deal with the first case, the second case being simpler.

As above we have a map $\varphi^s_*: {\mathcal H}^s_{\cp'_\Phi}
\to {\mathcal H}^s_{\cp'_\Psi}$ between the corresponding leaf
spaces of foliations in $\wt'$.
Here $\cp'_\Phi$ and $\cp'_\Psi$ are the projections of $\wt'$ to the
orbit spaces of $\wwp$ and $\widetilde{\Psi}$ respectively.
We clearly have

$$g^s_{\wt, \wt'}(\varphi^s(S)) \ = \ \varphi^s_*(f^s_{\wt, \wt'}(S))
\ \ \ \ {\rm and} \ \ \ \
g^s_{\wt, \wt'}(\varphi^s(S')) \  = \ \varphi^s_*(f^s_{\wt, \wt'}(S')).$$

\noindent
Hence, $g^s_{\wt,
\wt'}(\varphi^s(S))$ and $g^s_{\wt, \wt'}(\varphi^s(S)')$ lie in different elementary segments of
${\mathcal H}^s_\Phi(\wt')$. We have the following alternatives:

\begin{description}
  \item[If $f^s_{\wt, \wt'}$ preserves orientation] then $f^s_{\wt, \wt'}(S) \prec f^s_{\wt,
\wt'}(S')$. In other words, the elementary segment containing $f^s_{\wt, \wt'}(S)$ is above the
elementary segment containing $f^s_{\wt, \wt'}(S')$. Since $\varphi^s_*$
preserves the order between elementary segments, we obtain:
      $$g^s_{\wt, \wt'}(\varphi^s(S)) \ = \ \varphi^s_*(f^s_{\wt, \wt'}(S)) \ 
\prec \ \varphi^s_*(f^s_{\wt, \wt'}(S')) \ = \ g^s_{\wt, \wt'}(\varphi^s(S'))$$
      But in this case, $g^s_{\wt, \wt'}$ is also orientation preserving; hence $\varphi^s(S) \prec \varphi^s(S')$.
  \item[If $f^s_{\wt, \wt'}$ reverses orientation] then $f^s_{\wt, \wt'}(S') \prec f^s_{\wt, \wt'}(S)$, therefore:
      $$g^s_{\wt, \wt'}(\varphi^s(S')) = \varphi^s_*(f^s_{\wt, \wt'}(S')) \prec \varphi^s_*(f^s_{\wt, \wt'}(S)) = g^s_{\wt, \wt'}(\varphi^s(S))$$
      Since $g^s_{\wt, \wt'}$ reverses orientation, we deduce $\varphi^s(S) \prec \varphi^s(S')$.
\end{description}

In both cases, we have proved $\varphi^s(S) \prec \varphi^s(S')$.
Therefore, $\varphi^s$ preserves the total orders on ${\mathcal H}^s_{\cp_\Phi}$ and ${\mathcal H}^s_{\cp_\Psi}$; since we already know that it is a bijection, it follows that $\varphi^s$ is a homeomorphism.

We can reproduce the same argument, but this time for the negative itineraries, and this time involving the unstable leaves: we then obtain that $\varphi^u$ is a homeomorphism. Therefore,
the restriction of $\varphi$ to $\cp_\Phi$ is continuous.
As previously observed, it proves that $\varphi$ is continuous on $\oo_\Phi \setminus \Delta_\Phi$.

Now the continuity on the entire $\oo_\Phi$ follows easily: indeed, on the one hand, the space of ends of
$\oo_\Phi \setminus \Delta_\Phi$ (respectively $\oo_\Psi \setminus \Delta_\Psi$) is naturally the union of $\Delta_\Phi$
and the end $\infty$ of the plane $\oo_\Phi$. Since $\oo_\Phi$ is a two dimensional plane,
and $\Delta_\Phi$ is discrete, it now follows that the restriction of
$\varphi$ to $\oo_\Phi \setminus \Delta_\Phi$ admits a unique continuous extension $\bar{\varphi}: \oo_\Phi \to
\oo_\Psi$, mapping $\Delta_\Phi$ onto $\Delta_\Psi$. On the other hand, $\varphi$ maps bijectively $\Delta_\Phi$
onto $\Delta_\Psi$. The problem is to prove that $\bar{\varphi}$ and $\varphi$ coincide on $\Delta_\Phi$.
This follows easily from the fact that for every element $\theta$ of $\Delta_\Phi$, $\varphi$ maps an
open half leaf of the
stable leaf of $\theta$ onto an open half leaf of
the stable leaf of $\varphi(\theta)$, and that these stable leaves do not
accumulate at other elements of $\Delta_\Psi$.

We have proved that $\varphi: \oo_\Phi \to \oo_\Psi$ is continuous. The proof of the Lemma is completed.
\end{proof}

\begin{proposition}{(Isotopic equivalence)}
\label{pro.theoremD'}
The flows $\Phi$ and $\Psi$ are isotopically equivalent.
\end{proposition}

\begin{proof}{}
Given Lemma \ref{le:continuous} this follows from established results:
First, Haefliger \cite{Hae} showed that Lemma
\ref{le:continuous} implies that there is a homotopy equivalence
of $M$ which takes orbits of $\Phi$ to orbits of $\Psi$.
Ghys \cite{Gh} explained how to produce a homeomorphism with
the same properties. This was explicitly done by the first
author, Theorem of  \cite{Ba1}, for Anosov flows. Essentially
the same proof works for pseudo-Anosov flows. Now since this homeorphism is an homotopy equivalence, it is isotopic to the identity (\cite[Theorem $7.1$]{waldlarge}).
\end{proof}

\section{Model flows}
\label{top:models}

\subsection{Construction of model pseudo-Anosov flows}
\label{sub.construct}

In this section, we recall the construction in \cite{bafe}, section $7$, of model pseudo-Anosov flows. They are obtained from building blocks,
which are standard neighborhoods of
intrinsic elementary Birkhoff annuli. Such a neighborhood
is homeomorphic to
$[0,1] \times {\bf S}^1 \times [0,1]$
(with corresponding $(x,y,z)$ coordinates).

The Birkhoff annulus is $[0,1] \times {\bf S}^1 \times
\{ 1/2 \}$, where
$\{ 0 \} \times {\bf S}^1 \times \{ 1/2 \}$ and
$\{ 1 \} \times {\bf S}^1 \times \{ 1/2 \}$ are the only
closed orbits of the semiflow in this block and they
are the boundaries of the Birkhoff annulus (see figure \ref{box}).

\begin{figure}
\centeredepsfbox{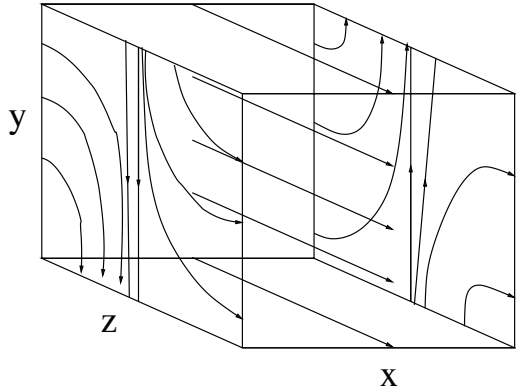}
\caption{Standard neighborhood of
a Birkhoff annulus.}
\label{box}
\end{figure}

The flow is tangent to the side boundaries
$\{ 0 \} \times {\bf S}^1 \times [0,1]$ and
$\{ 1 \} \times {\bf S}^1 \times [0,1]$.
The flow is incoming along
$[0,1] \times {\bf S}^1 \times \{ 0 \}$ and
outgoing along
$[0,1] \times {\bf S}^1 \times \{ 1 \}$.
The vertical orbits are the periodic orbits in
the block. For example, the stable manifold of
$\{ 0 \} \times {\bf S}^1 \times \{ 1/2 \}$ is
$\{ 0 \} \times {\bf S}^1 \times [0,1/2] $
and the unstable manifold is
$\{ 0 \} \times {\bf S}^1 \times [1/2,1]$.
In \cite{bafe} we prescribe an explicit formula for the
flow, denoted by $\Psi_\lambda$, in the block $B$, depending on a parameter $\lambda$.

Several copies of these blocks can be glued one to the other
along annuli which are half
of a tangential boundary annulus.
These are either the local stable or unstable
manifolds of one of the periodic orbits.
For example one
may glue
$\{ 0 \} \times {\bf S}^1 \times [0,1/2]$
to a similar half annulus in another copy of the block.
In particular we are glueing stable or unstable
manifolds of certain vertical orbits to similar
sets of other vertical orbits.
The glueings preserve the flow.
One can do this
in a very flexible way, so that in the end all
tangential boundary components are eliminated.
At this point one obtains
a semiflow in a manifold $P'$ which is a circle bundle over
a surface with boundary $\Sigma$. In order to define model flows,
we actually specify particular glueings between the stable/unstable annuli, so that every boundary component admits a natural coordinate system $(x,y)$ (\cite[Section 8]{bafe}).
In particular, the coordinate $y$ defines a function on $P'$,
whose level sets are sections of the fibration
over $\Sigma_i$. All this process is encoded by the data of a fat graph
$X$ embedded in $\Sigma$ satisfying the following properties:
\begin{enumerate}
  \item $X$ is a deformation retract of $\Sigma$.
  \item The valence of every vertex of $X$ is an even number.
  \item The set of boundary components of $\Sigma$ is partitioned
  in two subsets so that for every edge $e$ of $X$,
  the two sides of $e$ in $\Sigma$ lie in different sets of this partition.
  \item Each loop in $X$ corresponding to a boundary component of $\Sigma$ contains
  an even number of edges.
\end{enumerate}

We also allow Dehn surgery along the vertical
orbits so that the resulting manifold is a Seifert bundle over the surface
$\Sigma$. This operation is encoded by the data of the fat graph $(\Sigma, X)$
and also the Dehn surgery coefficients $D$ on vertical orbits (these coefficients
are well-defined with the convention that the meridians of the Dehn fillings to be the
loops contained in the section $\{y = Cte \}$ mentioned above). We denote
the resulting manifold with semi-flow by $(P(\Sigma, X, D), \Psi_\lambda)$.
Observe that $P(\Sigma, X, D)$ is a Seifert manifold.
Moreover, the $(x, y)$ coordinates in the initial building block provides
natural coordinates on every boundary component of $P(\Sigma, X, D)$. The last item
above ensures that each boundary component of $P(\Sigma, X, D)$ is a torus, as opposed to
being a Klein bottle.

Finally, we glue several copies $P_1$, ... , $P_k$ of such
Seifert fibered manifolds along their transversal boundaries.
Let us be more precise: for each $P_i$ and each component $T$ of
$P_i$, select a {\em vertical/horizontal basis} of $H_1(T, \mathbb Z)$, i.e.
a basis whose first element is represented by vertical loops
(i.e. regular fibers of $P_i$), and whose second element
is represented by the intersection between $T$ and the preferred section $\{ y = Cte \}$
we have defined above. One could think at first glance that we have defined by this way
a canonical basis of $H_1(T, \mathbb Z)$, but the point is that these homology classes are
defined only up to sign: $H_1(T, \mathbb Z)$ admits four vertical/horizontal basis.

Once such a basis is selected in each boundary torus, select
a pairing between these boundary tori, and for each such a pair $(T, T')$
choose
a two-by-two matrix $M(T,T')$ with integer
coefficients. It defines an isomorphism between $H_1(T, \mathbb Z)$ and $H_1(T', \mathbb Z)$;
hence an isotopy class of homeomorphisms between $T$ and $T'$. In
order to obtain a pseudo-Anosov flow in the resulting
manifold, it is necessary that the glueing maps do not map fibers
to curves homotopic to fibers, i.e. that no $M(T,T')$ is upper triangular.
In \cite{bafe} we show that the
the resulting flow, still denoted by $\Psi_\lambda$, is pseudo-Anosov
as soon as the real parameter $\lambda$ is sufficiently large.
We proved it for very particular
glueing maps between the boundary tori, which are linear in the natural
coordinates defined by the $(x,y)$ coordinates (even if the proof in \cite{bafe}
can be easily extended to less restrictive choices).
It is clear that for the resulting flow
$\Psi_\lambda$ - called a {\em model flow} - the spine in every piece $P_i$,
as stated in Theorem C, is the preimage
by the Seifert fibration of the fat graph $X_i$.

An immediate consequence of Theorem D' is that the resulting pseudo-Anosov flow
is insensitive up to isotopic equivalence to several choices: it does not depend on the real parameter $\lambda$, nor to the choice of the glueing map in the isotopy class defined by
the matrices $M(T,T')$, as long as these
data are chosen so that the resulting flow is pseudo-Anosov.

In summary, the model flow is uniquely defined up to isotopic equivalence by:

\begin{enumerate}
  \item The data of a family of fat graphs $(\Sigma_i, X_i)$ and
  Dehn filling coefficients $D_i$ ($i=1, ... ,k$).
  \item A choice of vertical/horizontal basis in each $H_1(T, \mathbb Z)$,
  \item A pairing between the boundary tori of the
  $P_i = P(\Sigma_i, X_i, D_i)$ (i.e. a pairing between the boundary components of
  the $\Sigma_i$'s).
  \item For each such a pair $(T, T')$, a two-by-two matrix with integer coefficients which is not upper triangular.
\end{enumerate}

Observe that item $(2)$ is not innocuous: Suppose that for some $P_i$ we replace the first vector in
the vertical/horizontal basis of each component of $P_i$ (but we do not modify the basis
of the boundary components of the other pieces $P_j$, nor any of the other combinatorial data).
Then we obtain a model flow on the same manifold, with the same spine decomposition and vertical
stable/unstable annuli, {\em but where the orientation of the vertical periodic orbits has been reversed.} Therefore, this second model pseudo-Anosov flow is {\em not} isotopically equivalent to the initial model flow.

We could formulate a statement establishing precisely when two initial combinatorial data provide topologically
equivalent model pseudo-Anosov flows, but it would require a detailed presentation of F. Waldhausen's classification
Theorem of graph manifolds (\cite{wald1, wald2}). We decided that it would be an unnecessary complication, and that Theorems D and D' already provide a convenient formulation for the solution of the classification problem of totally periodic pseudo-Anosov flows.
For example one issue we do not address is the choice of section in each
Seifert fibered piece minus small neighborhoods of singular fibers.
The boundary curves of these sections are essential to the classification of 
totally periodic pseudo-Anosov flows. These boundary curves are determined
up to Dehn twists along vertical curves
\cite{wald1,wald2,BNR}.

%This is a very flexible construction and
%Russ Waller \cite{Wa} showed that it is extremely common.

\subsection{Topological equivalence with model pseudo-Anosov flows }
\label{sub.isotopequiv}
In this section we present the proof of the Main Theorem, i.e. we show why every totally periodic pseudo-Anosov flow $(M, \Phi)$ is topologically equivalent to one of the model pseudo-Anosov flows
constructed in the previous section.

According to section \ref{sec:disj}, the manifold $M$ is obtained by glueing
the Seifert pieces $N(Z_i)$ along (transverse) tori. Moreover, every $N(Z_i)$ can be obtained by
glueing appropriate sides of a collection of ``blocks" $U(A_p)$, where each $A_p$ is a Birkhoff annulus. The union of these Birkhoff annuli is the spine $Z_i$
of $N(Z_i)$.
Every block $U(A_p)$ can be described as follows: its boundary contains two annuli, one inward and the other outward, whose lifts correspond to what have been called
elementary bands in the previous section. The block $U(A_p)$ also contains two periodic orbits (the boundary components of $A_p$; which may be identified through the glueing),
two stable annuli (one for each vertical periodic orbit) and two unstable annuli (also one for each periodic orbit) connecting the periodic orbits to the
entrance/exit transverse annuli in the boundary.
All these annuli constitute the boundary of $U(A_p)$.
%They are the vertical stable/unstable annuli in $N(Z_i)$.
The remaining part of $U(A_p)$ is a union of orbits
crossing $A_p$, and joining the entrance annulus to the exit annulus.
Of course, the lifts in $\mi$ of these blocks are nothing but what have been called \textit{blocks} in section~\ref{sub:stablestable}.

The entire Seifert piece $N(Z_i)$ is obtained by glueing all these $U(A_p)$ along
their stable and unstable sides.
The way this glueing has to be performed is encoded by a fat graph whose vertices correspond to
the vertical periodic orbits, and the edges are the Birkhoff annuli $A_p$. Let us be slightly more precise:
remove around every vertical periodic orbit a small tubular neighborhood. The result is a compact $3$-manifold $N(Z_i)^\ast$ which is a circle bundle over
a surface $\Sigma^\ast_i$ with boundary.
Notice that, as any other circle bundle (orientable or not) over a surface (orientable or not) with non-empty boundary, this circle bundle admits a section. In other words,
we can consider $\Sigma^\ast_i$ as a surface embedded in $N(Z_i)^\ast$; the edges of $X_i$ then are the intersection between $\Sigma_i^\ast$ and the the Birkhoff annuli $A_p$.

There are three types of boundary components:

\begin{itemize}
  \item boundary components corresponding to exit transverse tori,
  \item boundary components corresponding to entrance transverse tori,
  \item boundary components corresponding to (the boundary of tubular neighborhoods of) vertical periodic orbits.
\end{itemize}

Let us define the surface $\Sigma_i$ obtained by shrinking the last type of boundary components to points, that we call \textit{special points}.
One can select the circle fibration $\eta_i^\ast: N(Z_i)^\ast \to \Sigma_i^\ast$ so that the
restrictions of the Birkhoff annuli $A_p$ are
vertical. Then their projections define
in $\Sigma_i$ a collection of segments which are the edges of a graph $X_i$ embedded in $\Sigma_i$. Moreover, since every orbit in $N(Z_i)$ crosses a Birkhoff
annulus or accumulates on a vertical orbit, $X_i$ is a retract of $\Sigma_i$, in other words,
$(\Sigma_i, X_i)$ is a fat graph.
Observe that it satisfies the four properties
required in the definition of model flows in the previous section: we have just established
the first and third items; the second item corresponds to the fact that at each vertical periodic orbit there is an equal
number of stable and unstable vertical annuli, so that each of them is adjacent to an even number of Birkhoff annuli. The last item corresponds to the fact that the component of $\partial P'$ must be tori, not Klein bottles.

Therefore, there is a model semi flow $\Psi_i$ on a circle bundle $P'_i = P(\Sigma_i, X_i) \rightarrow \Sigma_i$ as
described in the previous section.
This partial flow has essentially the same properties that the restriction $\Phi_i$ of $\Phi$ to $N(Z_i)$
has: $P'_i$ is a circle bundle over $\Sigma_i$;
its boundary components are transverse to $\Psi_i$; moreover the inward (respectively outward) boundary components for $\Psi_i$
correspond to the boundary components of $\Sigma_i$ that are the entrance (respectively exit) components as defined previously, i.e. with respect
to $\Phi_i$. Moreover, the preimage in $P'_i$ of edges of $X_i$ are annuli $A^0_p$ transverse to $\Psi_i$. Every orbit of $\Psi_i$ either crosses
one (and only one) $A^0_p$, or accumulates on one vertical periodic orbit.
%Finally, the orientation of fibers over special points defined by $\Phi_i$ and $\Psi_i$ coincide.

The main difference is that $P'_i$ is a circle bundle over $\Sigma_i$,
whereas $N(Z_i)$ is merely a Seifert manifold with a
Seifert fibration.
In particular near the vertical periodic orbits, the fibration of $P_i$ is a product fibration.
Moreover, some of the vertical periodic orbit of $\Psi_i$ might be $1$-prong orbits.
There is an embedding $N(Z_i)^* \hookrightarrow P'_i$ preserving the fibers, and mapping
$\Sigma^\ast$, considered as a section in $N(Z_i)^\ast$, into
the canonical section $\{ y=Cte \}$ of $N(X_i) \to \Sigma_i$. We can furthermore choose
this embedding so that it is coherent relatively to the orientation of vertical orbits:
we require that the orientation of the regular fibers in each component $C$ of $\partial N(Z_i)^*$ defined by the oriented periodic orbit of $\Phi$ surrounded by $C$
coincide with the orientation defined by the periodic orbit of $\Psi_i$
surrounded by the image of $C$ in $P'_i$.

Although $N(Z_i)$ is not always homeomorphic to $P'_i$, it is obtained from it by a Dehn surgery along the vertical orbits.
This Dehn surgery is encoded by the data of Dehn coefficients $D_i$, i.e. the data at each vertex of $X_i$ of
a pair $(p_i, q_i)$ of relatively prime integers. This data is well-defined once
one selects the section $\Sigma_i^*$, since meridians around every vertical orbit can be defined as the loops contained in $\Sigma^*$. Hence, there is a homeomorphism between $N(Z_i)$ and the model
piece $P_i = P(\Sigma_i, X_i, D_i)$, which maps oriented vertical orbits of $\Phi$ contained
in $Z_i$ to oriented vertical orbits of the model semi-flow $\Psi_i$. Observe that
this homeomorphism maps stable/unstable vertical annuli into stable/unstable vertical annuli of $\Psi_i$.

Now $M$ is obtained by glueing exit transverse tori to entrance transverse tori of the various Seifert pieces $N(Z_i)$ through identification homeomorphisms
$\varphi_k$ (where $k$ describes the set of transverse tori $T_k$ as denoted previously).
The isotopy classes of these homeomorphisms can be characterized by two-by-two matrices,
once selected in each $H_1(T, \mathbb Z)$ vertical/horizontal basis, where the second (horizontal) element of the basis is now determined according to the section $\Sigma_i^\ast$.
These two-by-two matrices cannot be upper triangular, since adjacent Seifert pieces in a (minimal) JSJ decomposition cannot have freely homotopic regular fibers.

The choice of vertical/horizontal homological basis in each boundary torus of $N(Z_i)$
naturally prescribes a choice of vertical/horizontal homological basis in each boundary torus of $P_i \approx N(Z_i)$.

Hence we have collected all the necessary combinatorial data necessary for the construction
of a model flow $\Psi$ on a manifold $M_\Psi$ obtained by glueing the various pieces $P_i$.
The resulting manifold $M_\Psi$ is homeomorphic to $M$.

Now it should be clear to the reader that the $(M, \Phi)$ and $(M_\Psi, \Psi)$ have the same
combinatorial data, so that they are topologically equivalent
by theorem D'. This completes the proof of the Main Theorem.

\section{Concluding remarks}

Throughout this section $M$ is a graph manifold admitting
a totally periodic pseudo-Anosov flow.

\subsection*{Numbers of topological equivalence classes of pseudo-Anosov flows}

Waller \cite{Wa} proved that in general the number of fat graph structures on
a given topological surface $S = \Sigma_i$, even if always finite, can be quite big.
These all will generate topologically inequivalent pseudo-Anosov flows.
That is, different graphs $X_i$ yield different flows.
With the same fat graphs, with the careful choices of particular
sections and Dehn filling coefficients, made above,
then the following happens.
The only choice left is the direction of
the vertical periodic orbits.
Once the direction of a single vertical periodic orbit
in a fixed $P_i$ is chosen, then all the directions
in the other vertical orbits are determined, because the choice of orientations
propagates along Birkhoff annuli, i.e. edges of the fat graph.
But there are two choices here. So there are
$2^k$ inequivalent such pseudo-Anosov flows if there
are $k$ pieces in the JSJ decomposition of $M$.

Therefore, we have proved that there is no upper bound on the
number of topological equivalence classes of pseudo-Anosov flows
on $3$-manifolds.

\subsection*{Action of the mapping class group on the space of isotopic equivalence classes}

Since $M$ is toroidal its mapping class group is
infinite: it contains Dehn twists along embedded incompressible torus in $M$.
Let us be more precise: let $T$ be an embedded incompressible torus; and let $U(T)$ be a tubular neighborhood of $T$.
Since $M$ is orientable, $U(T)$ is
diffeomorphic to the product of an annulus $A$ by the circle ${\bf S}^1$. Let $\tau$ be a Dehn twist in the annulus $A$; then the map $\tau_T: A \times {\bf S}^1 \rightarrow A \times {\bf S}^1$ defined by $\tau_T(x, \theta) = (\tau(x), \theta)$ defines
a homeomorphism of $M$, with support contained in $U(T)$, which is not homotopically trivial
(since its action on the fundamental group of the incompressible torus $T$ is not trivial).
Notice in addition that there are infinitely many inequivalent ways of
expressing $U(T)$ as a product of an annulus and a circle.

Furthermore there may be symmetries of $M$. The JSJ
decomposition of $M$ is unique up to isotopies, so any
self homeomorphism of $M$ preserves the collection
of Seifert pieces \cite{Ja-Sh,Jo}, but conceivably could permute them,
and also in any piece it could have Dehn twist actions,
etc..

We conclude that
there are infinitely many isotopy equivalence classes of pseudo-Anosov flows in such
manifolds.

But the following remarkable property holds: Suppose that $T$ is a torus of the JSJ decomposition
and $\tau_T$ is a Dehn twist
in a {\em vertical} direction of $T$, i.e. for which the annnulus $A$
contains a loop freely homotopic to the regular fibers of one of the two Seifert pieces
bounded by $T$. Then $\tau_T$  does not change the isotopy class of totally periodic pseudo-Anosov flows,
i.e. the conjugate of the flow by any representant of $\tau_T$ is isotopically equivalent
to the initial flow.

Indeed, let $\Phi$ be any such a flow. Up to isotopic equivalence, one can assume that
$T$ is contained in one neighborhood $N(Z_i)$, in the region between $Z_i$ and a boundary component of $N(Z_i)$ isotopic to $T$. Then, we can select in the isotopy class of $\tau_T$
another homeomorphism $f$ with support disjoint from $Z_i$, disjoint from all the $N(Z_j)$ with $j \neq i$, and which preserves all the vertical
stable/unstable annuli in $N(Z_i)$. Therefore, the conjugate of $\Phi$ by $f$ has the same
combinatorial/topological data than $\Phi$, and thus, according to Theorem D', is isotopically equivalent to $\Phi$.

Observe that if $N(Z_i)$ is the other Seifert piece containing $T$ in its boundary, one can
obtain many other elements of the mapping class group in the stabilizer of the isotopic equivalence class by composing Dehn twists in the vertical direction for $N(Z_i)$ with
Dehn twists in the vertical direction for $N(Z_j)$.

\subsection*{Topological transitivity of totally periodic pseudo-Anosov flows}

It is easy to construct non transitive totally periodic pseudo-Anosov
flows. For simplicity we do an example without Dehn surgery.
Start with a surface $S$ with high enough genus
and three boundary components. Russ Waller \cite{Wa}
showed that $S$ has a structure as a fat graph with
the properties of section 5 and one can then easily
construct a semiflow $\Psi'$ in $S \times {\bf S}^1$,
with one exiting boundary component and two entering
components. Glue one entering boundary component with
the exiting one by an admissible glueing. This is
manifold $M_1$ with semiflow $\Psi_1$ with one
entering component. Do a copy of $M_1$ with a flow
reversal of $\Psi_1$ and then glue it to $M_1$
by an admissible map. The resulting manifold is
a graph manifold with a model flow which is
clearly not transitive: the boundary component
of $M_1$ is transverse to the flow and any orbit
intersecting this torus is trapped in $M_1$.

It is not very hard to characterize when exactly
the model flow is transitive: we claim that it is equivalent to
the following property:
the oriented graph $\mathfrak G$, which is the
quotient by $\pi_1(M)$ of the oriented graph $\widetilde{\mathfrak G}$ defined in section~\ref{su.realization},
is {\em strongly connected:} for any pair of vertex $T$ and $T'$ in $\mathfrak G$, there must be an oriented
path going from $T$ to $T'$ and another oriented path going from $T'$ to $T$.

This is done for a more
general class of flows (at least the Anosov case)
in a forthcoming article \cite{BBB} by B\'eguin, Bonatti and Yu.
Because of that we do not discuss further transitivity
of model flows here.

{\footnotesize
{
\setlength{\baselineskip}{0.01cm}

}
}

\end{document}